\documentclass[reqno,psamsfonts,11pt]{gen-j-l}
\usepackage{amssymb,amsmath,amsbsy,amsthm,esint,setspace}
\usepackage{hyperref}
\usepackage{amsrefs}
\usepackage{enumitem}
\usepackage{color}

\newtheorem{theorem}{Theorem}[section]
\newtheorem{lemma}[theorem]{Lemma}
\newtheorem{proposition}[theorem]{Proposition}

\newtheorem{remark}{Remark}[section]

\let\temp\phi
\let\phi\varphi
\let\varphi\temp

\begin{document}

\title[Moments for Boltzmann's Equation via Wigner Transform]
{Moments and Regularity for a Boltzmann
Equation via Wigner Transform}

\author{Thomas Chen, Ryan Denlinger, and Nata{\v s}a Pavlovi{\' c}}

\begin{abstract}
In this paper, we continue our study of the Boltzmann equation 
by use of tools originating from the analysis of dispersive equations in quantum dynamics. 
Specifically, we focus on properties of solutions to the Boltzmann equation 
with collision kernel equal to a constant in the spatial
domain $\mathbb{R}^d$, $d\geq 2$, which we use as a model in this paper. 
Local well-posedness
for this equation has been proven using the Wigner transform when
$\left< v \right>^\beta f_0 \in L^2_v H^\alpha_x$ for $\min (\alpha,\beta) > \frac{d-1}{2}$.
We prove that if $\alpha,\beta$ are large enough, then it is possible to propagate
moments in $x$ and derivatives in $v$ (for instance,
$\left< x \right>^k \left< \nabla_v \right>^\ell f \in L^\infty_T 
L^2_{x,v}$ if $f_0$ is nice enough).
The mechanism is an exchange of regularity in return
for moments of the (inverse) Wigner transform of $f$. 
We also prove a persistence of regularity result for the
scale of Sobolev spaces $H^{\alpha,\beta}$; and, 
continuity of the solution map in $H^{\alpha,\beta}$.
Altogether, these results allow us to conclude non-negativity of solutions, conservation of
energy, and the $H$-theorem for sufficiently regular solutions constructed via the
Wigner transform. Non-negativity in particular is proven to hold in $H^{\alpha,\beta}$ for any
$\alpha,\beta > \frac{d-1}{2}$, without any additional regularity or decay assumptions.
\end{abstract}

\maketitle

\section{Introduction}
\label{intro}

We are interested in the local Cauchy theory for the full Boltzmann equation:
\begin{equation}
\label{eq:boltz}
\left( \partial_t + v \cdot \nabla_x \right) f (t,x,v) =
Q (f,f) (t,x,v).
\end{equation}
Here $t \geq 0$, $x,v \in \mathbb{R}^d$ with $d\geq 2$, and the collision operator
$Q(f,f)$ is defined as follows:
\begin{equation}
Q (f,f) = \int_{\mathbb{R}^d} \int_{\mathbb{S}^{d-1}}
d\omega dv_* \mathbf{b} \left( \left| v-v_* \right|,
\omega \cdot \frac{v-v_*}{\left| v-v_* \right|} \right)
\left( f^\prime f_*^\prime - f f_* \right).
\end{equation}
We have defined $f^\prime = f(t,x,v^\prime)$, $f_*^\prime = f(t,x,v_*^\prime)$,
$f_* = f(t,x,v_*)$; the velocities $(v^\prime,v_*^\prime)$ are defined in terms
of $(v,v_*)$ and $\omega \in \mathbb{S}^{d-1}$ by the following relation:
\begin{equation}
\begin{aligned}
v^\prime & = v + (\omega \cdot (v_* - v)) \omega \\
v_*^\prime & = v_* - (\omega \cdot (v_* - v)) \omega.
\end{aligned}
\end{equation}
We will assume throughout this paper that the collision kernel
$\mathbf{b}$ is a bounded function or, for some results, identically constant.
These restrictive assumptions are for technical simplicity; they can certainly
be relaxed for most (and probably all) of our results. We refer to
\cite{CIP1994,Vi2002} for background on Boltzmann's equation.

Although in this paper we focus on properties of local in time solutions for Boltzmann's equation, 
we recall that 
there are two known theories of global solutions for Boltzmann's equation (apart from
small data solutions, which are closer to the local theory). One
 is the theory of renormalized solutions, which applies for arbitrary large data
having finite mass, second moments, and entropy. \cite{DPL1989} Very little is
known about renormalized solutions; in particular, it is not known whether they
are unique. The other theory of global solutions concerns solutions near a
Maxwellian equilibrium of some fixed temperature,
\cite{Uk1974,GS2011,AMUXY2011,Du2008,Gu2003,UY2006}.
The construction of such solutions is intimately tied to the properties of the linearized
collision operator. We are ultimately motivated by certain applications for which it
seems better to view Boltzmann's equation as a perturbation of free transport, rather
than a perturbation of the linearized equation. For example, in the derivation of
Boltzmann's equation from Hamiltonian particle systems \cite{L1975,Ki1975,GSRT2014},
it is very hard to make use of the structure arising from the linearized collision
operator (but see \cite{BGSR2015,BGSR2015III}); the transport structure is still available
in this context 
and that motivates us to employ it. For now, we do that in the setting of local in time solutions.

There are a number of theories of \emph{local} solutions for Boltzmann's equation
currently available; a local well-posedness (LWP)
 theory is a theory of existence and uniqueness which allows data of arbitrary size in some
norm, but which may break down after a short time depending (solely) on the norm of the data.
We refer the reader to \cite{KS1984,Ar2011,AMUXY-Loc} for several of the LWP theories which
are currently known for Boltzmann's equation. A new LWP theory for 
Boltzmann's equation has recently been developed in \cite{CDP2017} using the Wigner transform and tools 
including the bilinear Strachartz estimate
that are inspired by techniques for
the treatment of nonlinear Schr{\" o}dinger equations and the Gross-Pitaevskii 
hierarchy (see e.g. \cite{PC2010}). 
In the present paper,
we show that the solutions of Boltzmann's equation constructed\footnote{We emphasize that, contrary to \cite{Ar2011}, the theory of \cite{CDP2017} has
not been optimized to take advantage of the regularizing properties of the Boltzmann ``gain''
collision term. We expect such optimizations to be available through the use of more general
$X^{s,b}$ norms than considered in this paper; such refinements are the subject of ongoing research.}
in \cite{CDP2017} propagate higher regularity and moments if they are available
at the initial time. 
We give a complete description of this phenomenon
when the collision kernel is identically constant; the reasons for this
limitation seem to be purely technical.

The problem of regularity for Boltzmann's equation has been studied by various authors.
We refer especially to \cite{BD2000,Po1988}; both of these works discuss the 
persistence of regularity of small solutions near vacuum, as in \cite{KS1984}. 
It is also proven in \cite{BD2000} that, for some Boltzmann equations with
Grad cut-off, in the case of small solutions, the solution at positive times propagates
the singularities of the initial data. This shows that, in the Grad cut-off case,
the Boltzmann flow does not smooth out irregular initial data. For this reason, in the
present work we cannot hope to prove that the solution is \emph{smoother} than the
data; we can only hope to show  that the solution is \emph{as regular} as the
data 
and that is what we achieve. More precisely, 
we obtain a fairly complete description
of the persistence of regularity for local-in-time solutions
using the functional framework of \cite{CDP2017}, at least when the collision kernel
is constant.

The study of moments for the Boltzmann equation has a long history, particularly in the
space homogeneous case. We refer to \cite{Vi2002} for a review of the classical
results, as well as to e.g. \cite{bo97, gapavi09, lumo12, alcagamo13, algapata15}
and references therein. 
The rough picture is that, for hard potentials (including hard spheres), the
space homogeneous Boltzmann equation generates higher-order moments \emph{instantaneously}
as soon as the initial energy is finite. For Maxwell molecules or soft potentials, however,
only those moments which are initially finite are finite at positive times. In particular,
in the case of \emph{bounded} collision kernels (which is the only case studied in this paper),
one does not expect any generation of moments effect. Instead, one should seek to prove
\emph{propagation} of moments, which is precisely the type of result we can prove.
Our techniques could be applied in
the case of hard potentials with Grad cut-off, but at present we cannot expect to capture
the generation of moments effect in our estimates. 
We also remark that we  address $L^2$ moments, 
whereas the space homogeneous theory is primarily concerned
with $L^1$ moments.\footnote{There is nothing particularly special about $L^1$ moments;
for example, a theorem of Lanford on the derivation of Boltzmann's equation
\cite{L1975} relies upon exponential $L^\infty$ moments. 
}

As is typically the case when one proves a propagation result, the proofs 
in this paper are based on the following idea:
we write an equation for a desired moment or derivative, and then solve that equation by a fixed point argument
as in \cite{CDP2017}.\footnote{Technically one must show that the solution of the new fixed
point problem is related to the solution the original fixed point problem in the expected
way; this is formally trivial but inconvenient to prove rigorously. 
We shall not address this issue in full detail.}
 In all cases we must pay a cost in terms of regularity/moments
in order to propagate regularity/moments in some other variable. 
 For example, we can propagate moments in $x$ by paying with
moments in $v$ until we run out of currency to exchange; this results in a natural limit
to the number of moments we can propagate. Similarly, for an \emph{identically constant}
collision kernel, we can propagate derivatives in $v$ by paying with derivatives
 in $x$, until we run out of derivatives to trade. Now it is \emph{a priori} possible that
solutions which are initially smooth in $x$ with rapid decay in $v$ lose this property
after some short time, only to persist for a longer time in a less regular space. We
can rule this out, to some extent, by our methods as well. In particular, we prove
persistence of regularity results for the scale of Sobolev spaces which arise naturally
from the analysis of \cite{CDP2017}. Effectively, as soon as the solution has enough regularity
and decay to apply the methods of \cite{CDP2017}, we can prove that \emph{more regular}
data leads to more regular solutions for as long as the solution exists in the
less regular space. The higher Sobolev norms may grow much more rapidly in time than the lower
Sobolev norms, however.

As an application of our results, we are able to prove that for constant collision kernels,
the solutions constructed in \cite{CDP2017} propagate \emph{non-negativity} assuming only
that the data itself is non-negative. This is true under minimal assumptions for which the
theory of \cite{CDP2017} applies; in particular, we \emph{do not} require higher moment
or regularity estimates to prove non-negativity. The reason is that our persistence of 
regularity results allow us to approximate low-regularity solutions by higher regularity
solutions for a short time interval; even though the higher Sobolev norms may be very
large, they will at least be finite on a time interval bounded from below uniformly with
respect to the mollification parameter. Once we have enough regularity and decay, we can
apply a theorem of \cite{LZ2001} directly (that theorem is based on a Gronwall type
argument to control the negative part of the solution). The non-negativity is preserved
upon passage to the limit.

\textbf{Organization of the paper.} 
In Section \ref{sec:Ntn}, we outline the basic notation and the main results we 
prove in this paper. In Section \ref{sec:Prop}, we quote a slightly refined version of
a key proposition from \cite{CDP2017} which we require to prove our main results;
we also prove a couple of useful lemmas. Section \ref{sec:Mts} is dedicated to proving
propagation of moments in $x$ and derivatives in $v$. In Section \ref{sec:pers}, we
prove persistence of regularity in the scale of Sobolev spaces $H^{\alpha,\beta}$; this
corresponds to derivatives in $x$ and moments in $v$. In Section \ref{sec:treg}, we
address regularity with respect to the time variable. Section \ref{sec:cont} contains a proof
of continuity of the solution map. A new bilinear estimate with loss is proven in
Appendix \ref{sec:AppA}; this estimate is used in Section \ref{sec:treg}.

\section*{Acknowledgements}
\label{sec:acknowledgements}

R.D. gratefully acknowledges support from a postdoctoral fellowship
at the University of Texas at Austin. N.P. is funded in part by NSF DMS-1516228.
T.C. gratefully acknowledges support by the NSF through grants DMS-1151414 (CAREER) and DMS-1716198.

\section{Notation and Main Results}
\label{sec:Ntn}

In this section we present the notation, followed by a brief review of the relevant previous works and the statement of 
main results of this paper. 

\subsection{Notation}

We will employ the Wigner transform as in \cite{CDP2017}. For any function
$f(x,v) \in L^2_{x,v}$, we define $\gamma (x,x^\prime) \in L^2_{x,x^\prime}$ as follows:
\begin{equation}
\gamma (x,x^\prime) = \int_{\mathbb{R}^d} f\left( \frac{x+x^\prime}{2},v\right)
e^{iv\cdot (x-x^\prime)} dv.
\end{equation}
The inverse of this formula is the Wigner transform, given by:
\begin{equation}
f(x,v) = \frac{1}{(2\pi)^d} \int_{\mathbb{R}^d}
\gamma \left( x+\frac{y}{2},x-\frac{y}{2} \right) e^{-i v\cdot y} dy.
\end{equation}
We may write $f = \mathcal{W} [\gamma]$ and $\gamma = \mathcal{W}^{-1} [f]$. The map
$\mathcal{W} : L^2_{x,x^\prime} \rightarrow L^2_{x,v}$ is an isometric linear isomorphism
by Plancharel's theorem. We will assume throughout this paper that
\begin{equation}
\left\Vert \mathbf{b} \right\Vert_\infty
= \sup_{u \in \mathbb{R}^d,\; \omega\in \mathbb{S}^{d-1}}
\left| \mathbf{b} \left( |u|,\omega\cdot \frac{u}{|u|} \right) \right| < \infty.
\end{equation}
For some results, we will need to assume (for technical reasons)
that $\mathbf{b}$ is identically constant.
We define the Fourier transform of the collision kernel, namely:
\begin{equation}
\hat{\mathbf{b}}^\omega (\xi) = \int_{\mathbb{R}^d}
\mathbf{b} \left( |u|,\omega \cdot \frac{u}{|u|} \right) e^{-i u \cdot \xi}
du.
\end{equation}

The functional setting is the same as that of \cite{CDP2017}, and it is defined via the
Fourier transform of $\gamma = \mathcal{W}^{-1} [f]$:
\begin{equation}
\hat{\gamma} (\xi,\xi^\prime) = \int_{\mathbb{R}^d \times \mathbb{R}^d}
e^{-i x \cdot \xi} e^{-i x^\prime \cdot \xi^\prime}
\gamma (x,x^\prime) dx dx^\prime.
\end{equation}
For any $\alpha,\beta \geq 0$ we define
\begin{equation}
\left\Vert \gamma (x,x^\prime) \right\Vert_{H^{\alpha,\beta}} =
\left\Vert \left< \xi+\xi^\prime \right>^\alpha
\left< \xi-\xi^\prime \right>^\beta
\hat{\gamma} (\xi,\xi^\prime) \right\Vert_{L^2_{\xi,\xi^\prime}}.
\end{equation}
This norm is equivalent to the following norm for $f(x,v)$:
\begin{equation}
\left\Vert \left< 2v \right>^{\beta} (1-\Delta_x)^{\frac{\alpha}{2}}
f (x,v) \right\Vert_{L^2_{x,v}}.
\end{equation}

\subsection{Previous results}

If $f(t,x,v)$ is a smooth and rapidly decaying solution of Boltzmann's equation (\ref{eq:boltz})
with $\left\Vert \mathbf{b} \right\Vert_\infty < \infty$,
it is possible to show that $\gamma (t) = \mathcal{W}^{-1} [f(t)]$ solves the following
equation: (see Appendix A of \cite{CDP2017} for a proof)
\begin{equation}
\left( i \partial_t + \frac{1}{2} (\Delta_x - \Delta_{x^\prime}) \right) \gamma(t) 
= B (\gamma(t),\gamma(t))
\end{equation}
\begin{equation}
B(\gamma_1,\gamma_2) = B^+ (\gamma_1,\gamma_2) - B^- (\gamma_1,\gamma_2)
\end{equation}
\begin{equation}
\begin{aligned}
& B^- (\gamma_1,\gamma_2) (x,x^\prime) = \frac{i}{2^{2d} \pi^d} 
\int_{\mathbb{S}^{d-1}} d\omega \int_{\mathbb{R}^d} dz \hat{\mathbf{b}}^\omega \left( \frac{z}{2}\right)
\times \\
& \qquad \qquad 
\times \gamma_1 \left( x-\frac{z}{4},x^\prime + \frac{z}{4} \right)
\gamma_2 \left( \frac{x+x^\prime}{2}+\frac{z}{4},
\frac{x+x^\prime}{2}-\frac{z}{4}\right)
\end{aligned}
\end{equation}
\begin{equation}
\begin{aligned}
& B^+ (\gamma_1,\gamma_2) (x,x^\prime) = \frac{i}{2^{2d} \pi^d}
\int_{\mathbb{S}^{d-1}} d\omega \int_{\mathbb{R}^d} dz \hat{\mathbf{b}}^\omega \left( \frac{z}{2}\right)
\times \\
& \times \gamma_1 \left( x - \frac{1}{2} P_\omega (x-x^\prime) - \frac{R_\omega (z)}{4},
x^\prime + \frac{1}{2} P_\omega (x-x^\prime) + \frac{R_\omega (z)}{4} \right) \times \\
& \times \gamma_2 \left( \frac{x+x^\prime}{2} + \frac{1}{2} P_\omega (x-x^\prime)
+ \frac{R_\omega (z)}{4}, \frac{x+x^\prime}{2} - \frac{1}{2}
P_\omega (x-x^\prime) - \frac{R_\omega (z)}{4} \right)
\end{aligned}
\end{equation}
Here we have
\begin{equation}
P_\omega (x) = (\omega\cdot x ) \omega
\end{equation}
\begin{equation}
R_\omega (x) = (\mathbb{I}-2P_\omega) (x)
\end{equation}
and $\mathbb{I}(x)=x$. 

We will refer to $f$ as a solution of Boltzmann's equation if $\gamma
= \mathcal{W}^{-1} [f]$ solves the Duhamel-type integral formula:
\begin{equation}
\label{eq:gamma-boltz}
\gamma (t) = e^{\frac{1}{2} i t \Delta_{\pm}} \gamma(0) - i \int_0^t
e^{\frac{1}{2} i (t-t_1) \Delta_{\pm}} B(\gamma (t_1),\gamma(t_1)) dt_1
\end{equation}
We may also say in this case that $\gamma$ solves Boltzmann's equation.
Note that if $\mathbf{b}\equiv \textnormal{cst.}$ then $\hat{\mathbf{b}}^\omega$ is a 
$\delta$-function concentrated at $z=0$ (cf. Bobylev's formula).

\begin{remark}
We anticipate that some parts of the analysis of Boltzmann's
equation via  bilinear Strichartz estimates can be presented in terms of the
kinetic transport operator
\begin{equation}
\label{eq:dtvdx}
\partial_t + v \cdot \nabla_x
\end{equation}
Indeed, starting from (\ref{eq:dtvdx}) and
 taking the Fourier transform in $t$ and $x$, with dual variables
$\tau$ and $\eta$ respectively, yields the weight
\begin{equation}
\tau + v \cdot \eta
\end{equation}
To see the equivalence, replace $v \mapsto (\xi - \xi^\prime)/2$ and
$\eta \mapsto \xi + \xi^\prime$, and thereby recover the difference of squares
\begin{equation}
v\cdot \eta \mapsto \frac{1}{2} \left( |\xi|^2 - |\xi^\prime|^2 \right)
\end{equation}
which appears
in the spacetime Fourier analysis of the
 density matrix formulation of Schr{\" o}dinger's equation. This observation would be
especially helpful for problems with periodic boundary conditions, where the interpretation
of the Wigner transform is less clear.
\end{remark}

Compared to the kinetic formulation of Boltzmann's equation,
the density matrix formulation has some unique advantages; most importantly, it facilitates a
direct comparison to the substantial literature on Schr{\" o}dinger's equation in density matrix
formulation.\footnote{A detailed analysis of the spectral properties of $\gamma$ (viewed as
a linear operator on $L^2 (\mathbb{R}^d)$), or the
connections to Bobylev's formula, may also provide interesting
 avenues for research; however, we do not
explore these directions in this work.} This is particularly evident in the case of a constant
collision kernel $\mathbf{b} \equiv \textnormal{cst.}$,
 where the Boltzmann loss operator becomes (up to a real constant)
\begin{equation}
B^- (\gamma,\gamma) = i \rho_\gamma \left( \frac{x+x^\prime}{2}\right) \gamma (x,x^\prime)
\end{equation}
where $\rho_\gamma (x) = \gamma (x,x)$ is the diagonal of a density matrix.
Though this is unfortunately not a commutator, it does look very similar to the
bilinear operator
\begin{equation}
B (\gamma,\gamma)
= \left( \rho_\gamma (x) - \rho_\gamma (x^\prime) \right) \gamma (x,x^\prime)
 = \left[ \rho_\gamma , \gamma \right] (x,x^\prime)
\end{equation}
which defines the density matrix formulation of cubic NLS (see \cite{KM2008}).
Moreover, just as
$(d-1)/2$ is $L^2$-based Sobolev regularity threshold for our proof (in both
$\alpha$ and $\beta$ for a constant collision kernel), the regularity threshold
for the proof of \cite{KM2008} in the cubic NLS case is exactly $(d-1)/2$. This is not
surprising because our argument is a translation 
of the arguments of
\cite{KM2008} to the Boltzmann equation.
Note that, in the Boltzmann case, a somewhat better regularity class has been obtained 
for some collision kernels with Grad cut-off
using the Strichartz estimates of Castella and Perthame 
\cite{Ar2011,CP1996}, but it was necessary to use the convoluting effects of the Boltzmann
gain operator. We
 have not made essential use of the special properties of the gain term in
the present work.

The following LWP result was proven in \cite{CDP2017}:
\begin{theorem}
\label{thm:lwp}
Let $\alpha,\beta \in \left( \frac{d-1}{2},\infty\right)$ and consider the Boltzmann
equation with $\left\Vert \mathbf{b} \right\Vert_\infty < \infty$. For any
$\gamma_0 \in H^{\alpha,\beta}$ there exists a unique solution $\gamma(t)$ of Boltzmann's
equation on a small time interval $[0,T]$ such that
\begin{equation}
\left\Vert \gamma \right\Vert_{L^\infty_T H^{\alpha,\beta}} < \infty
\end{equation}
and 
\begin{equation}
\left\Vert B(\gamma,\gamma) \right\Vert_{L^1_T H^{\alpha,\beta}} < \infty
\end{equation}
both hold, and $\gamma(0) = \gamma_0$. Moreover, if $\left\Vert \gamma_0 \right\Vert_{H^{\alpha,\beta}}
\leq M$ then for all small enough $T$ depending only on $\alpha,\beta$ and $M$, there holds:
\begin{equation}
T^{\frac{1}{2}} \left\Vert \gamma \right\Vert_{L^\infty_T H^{\alpha,\beta}}
+ \left\Vert B(\gamma,\gamma) \right\Vert_{L^1_T H^{\alpha,\beta}} \leq
C(M,\alpha,\beta) T^{\frac{1}{2}} \left\Vert \gamma_0 \right\Vert_{H^{\alpha,\beta}}
\end{equation}
\end{theorem}
\begin{remark}
Note carefully that, as a consequence of Theorem \ref{thm:lwp}, 
$$\left\Vert B(\gamma,\gamma)
\right\Vert_{L^1_T H^{\alpha,\beta}}
$$ 
scales at worst like a power of $T$, namely 
$T^{\frac{1}{2}}$, when $T$ is small. We will make use of this fact repeatedly in this
work.
\end{remark}

The functional spaces $H^{\alpha,\beta}$ automatically guarantee some regularity in space and
decay in velocity variables. Our first result, to be proven in Section \ref{sec:Mts}, states that
it is possible to trade moments in $v$ for moments in $x$; and, if $\mathbf{b} \equiv 
\textnormal{cst.}$, it is also possible to trade
derivatives in $x$ for derivatives in $v$. For these results to hold, we must always assume that
enough decay or regularity is available in the initial data. We remind the reader that, for
bounded collision kernels, the Boltzmann equation is not expected to generate derivatives or
moments in any variable at positive times; hence, any regularity or decay estimate
 must have been present
at the initial time for it to be available at positive times.

\subsection{Main results of this paper}

Now we are ready to state main results of this paper which are formulated in Theorem \ref{thm:mts}, Theorem \ref{thm:pers}
and Theorem \ref{thm:nonneg}.

\begin{theorem}
\label{thm:mts}
Let $\gamma (t)$ be a solution of Boltzmann's equation with bounded collision kernel,
$\left\Vert \mathbf{b} \right\Vert_\infty < \infty$, satisfying the following bounds on some
time interval $[0,T]$:
\begin{equation}
\left\Vert \gamma(t) \right\Vert_{L^\infty_T H^{\alpha,\beta}} < \infty
\end{equation}
\begin{equation}
\left\Vert B(\gamma(t),\gamma(t)) \right\Vert_{L^1_T H^{\alpha,\beta}} < \infty
\end{equation}
with $\alpha,\beta > \frac{d-1}{2}$. Then we have the following: \\
(i) Suppose that, for some integer $K>0$ with $K< \beta - \frac{d-1}{2}$, for any integer $k$ with
$1\leq k \leq K$ there holds
\begin{equation}
\left< x+x^\prime \right>^k \gamma (0) \in H^{\alpha,\beta-k}
\end{equation}
Then for all $1 \leq k \leq K$ we have
\begin{equation}
\left< x+x^\prime \right>^k \gamma(t) \in L^\infty_T H^{\alpha,\beta-k}
\end{equation}
\\
(ii) Assume that $\mathbf{b} \equiv \textnormal{cst.}$; then,
suppose that, for some integer $K>0$ with $2K < \alpha - \frac{d-1}{2}$, for all $1\leq k \leq K$
there holds
\begin{equation}
\left< x-x^\prime \right>^{2k} \gamma (0) \in H^{\alpha-2k,\beta}
\end{equation}
Then for all $1\leq k \leq K$ we have
\begin{equation}
\left< x-x^\prime \right>^{2k} \gamma(t) \in L^\infty_T H^{\alpha-2k,\beta}
\end{equation}
\end{theorem}
\begin{remark}
Our proof provides quantitative estimates for $\left< x + x^\prime \right>^k \gamma (t)$ and
$\left< x-x^\prime \right>^{2k} \gamma(t)$ in the relevant function spaces. However, these
bounds may grow very rapidly with time and we make no effort to prove optimal bounds on
the growth rate.
\end{remark}

In order to apply Theorem \ref{thm:mts}, we generally require solutions $\gamma \in
H^{\alpha,\beta}$ for large values of $\alpha,\beta$. This naturally leads us to inquire
whether a solution may blow up in $H^{\alpha,\beta}$ only to persist longer in some less
regular space. This question is particularly relevant if we want to approximate some irregular
initial data by some other, very regular, data for the purpose of formal computation. 
Our next result, proven in Section \ref{sec:pers},
 rules out such pathological behavior within the scale of Sobolev
spaces $H^{\alpha,\beta}$ with $\alpha,\beta > \frac{d-1}{2}$.

\begin{theorem}
\label{thm:pers}
Let $\gamma(t)$ be a solution of Boltzmann's equation with $\left\Vert \mathbf{b} \right\Vert_\infty
< \infty$, and suppose $\gamma \in L^\infty_T H^{\alpha,\beta}$ and
$B(\gamma,\gamma) \in L^1_T H^{\alpha,\beta}$ for some $\alpha,\beta > \frac{d-1}{2}$.
Then we have the following: \\
(i) If $\gamma(0) \in H^{\alpha+r,\beta}$ for some $r\in \mathbb{N}$, then
$\gamma \in L^\infty_T H^{\alpha+r,\beta}$ and $B(\gamma,\gamma) \in L^1_T H^{\alpha+r,\beta}$. \\
(ii) If $\gamma(0) \in H^{\alpha,\beta+r}$ for some (real) $r>0$, then
$\gamma \in L^\infty_T H^{\alpha,\beta+r}$ and $B(\gamma,\gamma) \in L^1_T H^{\alpha,\beta+r}$.
\end{theorem}
\begin{remark}
As with Theorem \ref{thm:mts}, we can extract quantitative estimates from the proof of
Theorem \ref{thm:pers}, but they may grow very rapidly with time.
\end{remark}

We also prove, in Sections \ref{sec:treg} and \ref{sec:cont} respectively, results on regularity in
time and also the continuity of the solution map.\footnote{Note in particular that the proof in
Section \ref{sec:treg} relies upon the bilinear estimates \emph{with loss} proven in Appendix 
\ref{sec:AppA}; we emphasize that those bilinear estimates cannot replace
Proposition \ref{prop:KeyProp} elsewhere in this paper, regardless of the size of $\alpha,\beta$.}
We quote those results here for the convenience of the reader.
\begin{proposition}
\label{prop:treg-0}
Let $\gamma(t)$ be a solution of Boltzmann's equation with $\left\Vert \mathbf{b} \right\Vert_\infty
< \infty$, and suppose $\gamma \in L^\infty_T H^{\alpha,\beta}$ and
$B(\gamma,\gamma) \in L^1_T H^{\alpha,\beta}$ for some $\alpha,\beta > \frac{d-1}{2}$.
Further suppose that $K>0$ is an integer with $K < \min(\alpha,\beta)-\frac{d}{2}$.
 Then for any integer $k$ with
$1\leq k \leq K$ there holds
$\partial_t^k \gamma \in L^\infty_T H^{\alpha-k,\beta-k}$.
\end{proposition}

\begin{proposition}
\label{prop:cont-0}
Let $\gamma^j (t)$ be a solution of Boltzmann's equation with $\left\Vert \mathbf{b}
\right\Vert_\infty < \infty$, for $j=1,2$, with
$\gamma^j \in L^\infty_T H^{\alpha,\beta}$ and
$B(\gamma^j,\gamma^j) \in L^1_T H^{\alpha,\beta}$ for $j=1,2$ and some
$\alpha,\beta > \frac{d-1}{2}$. Furthermore, suppose that
$\left\Vert \gamma^j \right\Vert_{L^\infty_T H^{\alpha,\beta}} \leq M$ for $j=1,2$. Then we
have
\begin{equation}
\left\Vert \gamma^1 - \gamma^2 \right\Vert_{L^\infty_T H^{\alpha,\beta}}
\leq C_{M,T} \left\Vert \gamma^1 (0) - \gamma^2 (0) \right\Vert_{H^{\alpha,\beta}}
\end{equation}
where the constant may depend on $\alpha,\beta$.
\end{proposition}

\begin{remark}
Most likely, the proofs of the preceding theorems can be combined to prove propagation of
mixed derivatives and moments; for example,
\begin{equation}
\left< x+x^\prime \right>^{k_1} \left< x-x^\prime \right>^{k_2}
\left< \nabla_x + \nabla_{x^\prime} \right>^{k_3}
\left< \nabla_x - \nabla_{x^\prime} \right>^{k_4}
\partial_t^{k_5} \gamma (t)
\end{equation}
We will not address these mixed estimates in detail. For what follows, it will suffice to
notice that by Fourier transforming in time (only), we can always estimate
\begin{equation}
\left< x+x^\prime \right>^k \left< x-x^\prime \right>^k \partial_t^k
\lesssim
\left< x+x^\prime \right>^{3k} + \left< x-x^\prime \right>^{3k} + \partial_t^{3k}
\end{equation}
This is obvious in $L^2_t L^2_{x,x^\prime}$ but the same estimate holds in $L^2_t H^{\alpha,\beta}$ 
as well, at least
for integer values of $\alpha$ and $\beta$. In order to apply this estimate in practice, we must
use a smooth compactly supported cut-off in the time variable; to this end, it is helpful to
solve Boltmann's equation \emph{backwards} in time on a short time interval. This allows us
to perform estimates on $[0,T_1]$ (with $T_1 < T$) by choosing a cut-off which is supported on
$[-\Delta, T]$ for sufficiently small $\Delta > 0$. Due to the Grad cut-off condition,
there is no difficulty in solving Boltzmann's equation backwards for a short time interval.
\end{remark}

Using Theorem \ref{thm:mts} and Proposition \ref{prop:treg-0},
 we can construct solutions $\gamma(t)$ which have very strong
regularity and decay properties on a short time interval. For such solutions, we can reverse
the steps from Appendix A of \cite{CDP2017} and thereby show that
$f = \mathcal{W} [\gamma]$ is a \emph{classical} solution of Boltzmann's equation. In particular,
conservation of mass, momentum, and energy follow by the usual computations. In view of the
next result on non-negativity, we can also prove the $H$-theorem for solutions having enough
regularity and decay, under the additional assumption that
$f(0) = \mathcal{W} [\gamma(0)] \geq c e^{-c(|x|^2 + |v|^2)}$ for some $c>0$. Obviously we could
optimize the spaces in which energy conservation holds by density arguments; we will not
state a precise result along these lines.

One very important issue which was not addressed in \cite{CDP2017} was the \emph{non-negativity}
of solutions. Only non-negative solutions of Boltzmann's equation are considered to have
physical meaning. Moreover, the conserved mass and energy only supply useful control for
non-negative solutions; and, the entropy is only \emph{defined} for non-negative solutions.
Combining all of the results quoted in this section, we can prove the following:
\begin{theorem}
\label{thm:nonneg}
Let $\gamma(t)$ be a solution of Boltzmann's equation, with $\mathbf{b} \equiv \textnormal{cst.}$;
furthermore, suppose that $\gamma \in L^\infty_T H^{\alpha,\beta}$ and
$B(\gamma,\gamma) \in L^1_T H^{\alpha,\beta}$ for some $\alpha,\beta \in
\left( \frac{d-1}{2},\infty \right)$. Then if
$f (0,x,v) = \mathcal{W} [\gamma(0)] (x,v) \geq 0$ for almost every $x,v\in \mathbb{R}^d$, then
for all $t \in [0,T]$ we have $f(t,x,v) = \mathcal{W} [\gamma(t)] (x,v) \geq 0$ for almost
every $x,v \in \mathbb{R}^d$.
\end{theorem}
\begin{remark}
Note that in Theorem \ref{thm:nonneg} we do not require $\gamma$ to have any higher regularity
or moment bounds.
\end{remark}

We omit a complete proof of Theorem \ref{thm:nonneg}, but we will state a few remarks about
the proof. The first important point is that $\gamma(t)$ is actually  continuous in $H^{\alpha,\beta}$,
so we can  evaluate $\gamma(t)$ for \emph{any} $t\in [0,T]$. Fixing a solution
$\gamma \in L^\infty_T H^{\alpha,\beta}$ with $B(\gamma,\gamma) \in
L^1_T H^{\alpha,\beta}$, we define
\begin{equation}
T_1 = \sup \left\{ t_1 \in [0,T] \; : \;
f (t) = \mathcal{W}[\gamma(t)] \geq 0 \; \forall t \in [0,t_1] \right\}
\end{equation}
Assume by way of contradiction that $T_1 < T$. Since $\gamma(t)$ is continuous in time and
non-negativity is preserved under passage to the limit in $L^2$, we know that
$f (T_1) \geq 0$. Therefore, it suffices to propagate non-negativity on a \emph{small} time
interval (possibly much smaller than $T$). We pick a sequence of very regular functions
(say, in the Schwartz class) which converge to $\gamma (T_1)$ in $H^{\alpha,\beta}$; we use
these approximate functions as initial data in Boltzmann's equation.
Note carefully that the approximate solutions \emph{may not exist} on the full time
interval $[T_1,T]$, but they will have a time of existence which is bounded uniformly from below
due to uniform boundedness in $H^{\alpha,\beta}$. Since $f (T_1) \geq 0$, we can arrange for
the approximating functions to be non-negative at time $t=T_1$.
We can apply Theorem \ref{thm:pers},
Theorem \ref{thm:mts}, and Proposition \ref{prop:treg-0}
 to conclude that the approximating functions are smooth and rapidly
decaying for as long as they exist in $H^{\alpha,\beta}$; in particular, inverting the steps
from Appendix A of \cite{CDP2017}, we have a sequence of classical solutions of
Boltzmann's equation. We can apply the results of \cite{LZ2001} to conclude that the 
approximating sequence remains non-negative on a short time interval. Now we can pass to 
the limit, applying Theorem \ref{prop:cont-0}, to reach the desired contradiction.

\section{A Proposition and Two Lemmas}
\label{sec:Prop}

\begin{proposition}
\label{prop:KeyProp}
Suppose $\alpha,\beta \in \left( \frac{d-1}{2},\infty \right)$ and
 and
let $\delta \geq 0$ be chosen sufficiently small (with smallness depending continuously
on $\alpha,\beta,d$). Then there exists a constant $C$ (depending on $d,\alpha,\beta$)
such that if
$\left\Vert \mathbf{b} \right\Vert_\infty <\infty$ then for any 
$\gamma_1,\gamma_2 \in H^{\alpha,\beta}$, \textbf{both} the following estimates hold:
\begin{equation}
\label{eq:loss-est}
\begin{aligned}
& \left\Vert
B^- \left(
e^{\frac{1}{2} i t \left(\Delta_x - \Delta_{x^\prime}\right)}
\gamma_1,
e^{\frac{1}{2} i t \left( \Delta_x - \Delta_{x^\prime} \right)}
\gamma_2 \right)\right\Vert_{L^2_t H^{\alpha,\beta+\delta}} \leq \\
& \qquad \qquad \qquad \qquad \qquad \qquad \leq 
C \left\Vert \mathbf{b} \right\Vert_\infty
\left\Vert \gamma_1 \right\Vert_{H^{\alpha,\beta+\delta}}
\left\Vert \gamma_2 \right\Vert_{H^{\alpha,\beta}}
\end{aligned}
\end{equation}
\begin{equation}
\label{eq:gain-est}
\begin{aligned}
& \left\Vert
B^+ \left(
e^{\frac{1}{2} i t \left(\Delta_x - \Delta_{x^\prime}\right)}
\gamma_1,
e^{\frac{1}{2} i t \left( \Delta_x - \Delta_{x^\prime} \right)}
\gamma_2 \right)\right\Vert_{L^2_t H^{\alpha,\beta+\delta}} \leq \\
& \qquad \qquad \qquad \qquad \qquad \qquad \leq 
C \left\Vert \mathbf{b} \right\Vert_\infty
\left\Vert \gamma_1 \right\Vert_{H^{\alpha,\beta}}
\left\Vert \gamma_2 \right\Vert_{H^{\alpha,\beta}}
\end{aligned}
\end{equation}
\end{proposition}
\begin{remark}
Note carefully that the gain term $B^+$ regularizes in the $\beta$ index; the 
loss term, by contrast, exhibits no such regularization.
\end{remark}
\begin{proof} (case $\delta > 0$) 
\\
The case $\delta=0$ is proven in \cite{CDP2017}, or see Appendix \ref{sec:AppB},
 so we only have to consider $\delta>0$.
It turns out that the proofs are almost identical to the proof from \cite{CDP2017} so we
only sketch the ideas.

For the loss estimate (\ref{eq:loss-est}), we have (for example) the following commutativity:
\begin{equation}
\left( \nabla_x - \nabla_{x^\prime} \right) B^- (\gamma_1,\gamma_2) =
B^- \left( \left( \nabla_x - \nabla_{x^\prime} \right) \gamma_1 ,
\gamma_2 \right)
\end{equation}
and  moreover $\left( \nabla_x - \nabla_{x^\prime} \right)$ commutes with the free
propagator $e^{\frac{1}{2} i t (\Delta_x - \Delta_{x^\prime})}$. Hence the result is
immediately obtained for $\delta = 1$ once it is known for $\delta = 0$. The same result
can be proven for any $0 < \delta < 1$ using Fourier analysis, as in \cite{CDP2017};
the required modifications to the proof presented there are completely trivial.

For the gain estimate (\ref{eq:gain-est}), we note that all the estimates for the gain term from
\cite{CDP2017}, or Appendix \ref{sec:AppB},
 are stable with respect to a small perturbation of the \emph{target}
regularity index $\beta$ (keeping the regularities of $\gamma_1,\gamma_2$ fixed). In fact, due
to the \emph{angular averaging} effect, the
entire argument boils down to proving the convergence of certain geometric series of the
form $\sum_{k=1}^\infty 2^{-b k}$ for some $b >0$; obviously, the series will still
converge if we perturb $b$ slightly.
\end{proof}

\begin{lemma}
Consider the Boltzmann equation with arbitrary bounded collision kernel.
Then for any real numbers $a,b \geq 0$, there holds
\begin{equation}
\label{eq:Mts-Bm-p}
\left< x+x^\prime\right>^{a+b} B^- (\gamma_1,\gamma_2) =
B^- \left( \left< x+x^\prime \right>^a \gamma_1,
\left< x+x^\prime\right>^b \gamma_2 \right)
\end{equation}
\begin{equation}
\label{eq:Mts-Bp-p}
\left< x+x^\prime\right>^{a+b} B^+ (\gamma_1,\gamma_2) =
B^+ \left( \left< x+x^\prime \right>^a \gamma_1,
\left< x+x^\prime\right>^b \gamma_2 \right)
\end{equation}
In the case that the collision kernel $\mathbf{b} \equiv \textnormal{cst.}$,
we also have for any positive integer $k$, and any $a,b\geq 0$,
\begin{equation}
\label{eq:Mts-Bm-m}
\left< x-x^\prime\right>^a B^- (\gamma_1,\gamma_2) = 
B^- \left( \left<x-x^\prime\right>^a \gamma_1,
\left< x-x^\prime\right>^b \gamma_2 \right)
\end{equation}
\begin{equation}
\left( x-x^\prime \right) B^- (\gamma_1,\gamma_2) =
B^- \left( (x-x^\prime) \gamma_1, \gamma_2 \right)
\end{equation}
\begin{equation}
\begin{aligned}
\label{eq:Mts-Bp-m}
& \left< x-x^\prime \right>^{2k} B^+ (\gamma_1,\gamma_2) = \\
& \qquad \quad = 
\sum_{j_1 + j_2 + j_3 = k} \binom{k}{j_1,j_2,j_3}
(-1)^{j_3} 
B^+ \left( \left< x-x^\prime\right>^{2j_1} \gamma_1,
\left< x-x^\prime \right>^{2j_2} \gamma_2 \right)
\end{aligned}
\end{equation}
\begin{equation}
\label{eq:Mts-Bp-m-2}
\begin{aligned}
\left( x-x^\prime \right) B^+ (\gamma_1,\gamma_2) =
B^+ \left( (x-x^\prime) \gamma_1, \gamma_2 \right) +
B^+ \left( \gamma_1, (x-x^\prime) \gamma_2 \right)
\end{aligned}
\end{equation}
\end{lemma}
\begin{proof}
Only (\ref{eq:Mts-Bp-m}) requires some explanation. The difficulty is that
the action of $B^+$ involves the projection $P_\omega (x-x^\prime)$, which
does not disappear when taking \emph{differences} of $x$ and $x^\prime$.
This is easily dealt with, however, by using the following orthogonality
property:
\begin{equation}
\left< x-x^\prime \right>^{2k} =
\left( \left< (\mathbb{I}-P_\omega)(x-x^\prime)\right>^2 +
\left< P_\omega (x-x^\prime) \right>^2 - 1 \right)^k
\end{equation}
and expanding terms using the multinomial formula. For (\ref{eq:Mts-Bp-m-2}), we
use the simpler decomposition:
\begin{equation}
x-x^\prime = \left( \mathbb{I}-P_\omega \right) (x-x^\prime) + P_\omega (x-x^\prime)
\end{equation}
and conclude by linearity of the collision integral.
\end{proof}

\begin{lemma}
\label{lemma:deriv-commute}
Consider the Boltzmann equation with arbitrary bounded collision kernel; then
there holds
\begin{equation}
\left( \nabla_x + \nabla_{x^\prime} \right) B^-( \gamma_1,\gamma_2) =
B^- \left( \left( \nabla_x + \nabla_{x^\prime} \right) \gamma_1, \gamma_2 \right)
+ B^- \left( \gamma_1, \left( \nabla_x + \nabla_{x^\prime} \right) \gamma_2 \right)
\end{equation}
\begin{equation}
\left( \nabla_x + \nabla_{x^\prime} \right) B^+( \gamma_1,\gamma_2) =
B^+ \left( \left( \nabla_x + \nabla_{x^\prime} \right) \gamma_1, \gamma_2 \right)
+ B^+ \left( \gamma_1, \left( \nabla_x + \nabla_{x^\prime} \right) \gamma_2 \right)
\end{equation}
\end{lemma}
\begin{proof}
This is an elementary computation.
\end{proof}

\section{Moment Bounds for Density Matrices}
\label{sec:Mts}

In this section we treat propagation of moments of $\gamma$ in $x+x^\prime$ and
moments in $x-x^\prime$ in turn. Combining these results allows us to control
mixed moments as well, e.g. if we want to place
$\left< x \right> \left< \nabla_v \right> f$ in $L^\infty_T L^2_{x,v}$ we could use
\begin{equation}
\left< x+x^\prime \right> \left< x-x^\prime \right> \lesssim
\left< x+x^\prime \right>^2 + \left< x-x^\prime \right>^2
\end{equation}
More precise results for mixed moments may be available by combining the proofs
given in this section, but we do not pursue this issue in detail.

\subsection{Moments in $x+x^\prime$.} 

In this subsection we propagate moments of $\gamma$ in $x+x^\prime$, which is equivalent
to propagating moments of the distribution $f(x,v)$ in the spatial variable. The idea
of the proof is to write an equation for the $k$th moment and use the existence of
a solution $\gamma(t)$ in $H^{\alpha,\beta}$ for large enough $\beta$.

\begin{lemma}
\label{lemma:gkp-eq}
Consider a distributional solution $\gamma (t)$, $t \in [0,T]$, of
 the Boltzmann equation. Then for any
$k \in \mathbb{N}$, $k\geq 1$, there holds
\begin{equation}
\begin{aligned}
& \left( i \partial_t + \frac{1}{2} \left( \Delta_x - \Delta_{x^\prime}\right)\right)
\left( \left< x+x^\prime \right>^k \gamma (t) \right)
= B \left( \left< x+x^\prime \right>^k \gamma(t), \gamma(t) \right) + \\
& \qquad \qquad \qquad \qquad \qquad \qquad
+ k \frac{x+x^\prime}{\left< x+x^\prime \right>} \cdot
\left( \nabla_x - \nabla_{x^\prime}\right) \left(
\left< x+x^\prime \right>^{k-1} \gamma (t) \right).
\end{aligned}
\end{equation}
in the sense of distributions.
\end{lemma}
\begin{proof}
This computation follows by using  (\ref{eq:Mts-Bm-p}) and (\ref{eq:Mts-Bp-p}).
\end{proof}

Let us introduce the following convenient notation:
\begin{equation}
\zeta (t) = B \left( \gamma(t),\gamma(t) \right)
\end{equation}
\begin{equation}
\gamma_{k,+} (t,x,x^\prime) = \left< x+x^\prime \right>^k \gamma (t,x,x^\prime) 
\end{equation}
\begin{equation}
\zeta_{k,+} (t) = B \left( \gamma_{k,+} (t), \gamma (t) \right).
\end{equation}

\begin{proposition}
\label{prop:gkp-mts}
Let $\gamma (t)$ be a solution of Boltzmann's equation with bounded collision kernel,
$\left\Vert \mathbf{b} \right\Vert_\infty < \infty$, satisfying the following bounds
on some time interval $[0,T]$:
\begin{equation}
\label{eq:gammabd}
\left\Vert \gamma (t) \right\Vert_{L^\infty_T H^{\alpha,\beta}} < \infty
\end{equation}
\begin{equation}
\label{eq:zetabd}
\left\Vert B ( \gamma (t),\gamma (t) \right\Vert_{L^1_T H^{\alpha,\beta}} < \infty
\end{equation}
with $\alpha,\beta > \frac{d-1}{2}$. Further assume that, for some integer $K>0$ with
$K < \beta - \frac{d-1}{2}$, for all $1 \leq k \leq K$ there holds
\begin{equation}
\left\Vert \gamma_{k,+} (0) \right\Vert_{H^{\alpha,\beta-k}} < \infty
\end{equation}
Then for all $1\leq k \leq K$ we have
\begin{equation}
\left\Vert \gamma_{k,+} (t) \right\Vert_{L^\infty_T H^{\alpha,\beta-k}} < \infty.
\end{equation}
Moreover there is an explicit bound on
$\underset{1\leq k \leq K}{\sup}
\left\Vert \gamma_{k,+} (t) \right\Vert_{L^\infty_T H^{\alpha,\beta-k}}$ that only
depends on
$\left\Vert \gamma(t) \right\Vert_{L^\infty_T H^{\alpha,\beta}}$,
$\left\Vert B(\gamma(t),\gamma(t))\right\Vert_{L^1_T H^{\alpha,\beta}}$, and
$\underset{1\leq k \leq K}{\sup}
\left\Vert \gamma_{k,+} (0) \right\Vert_{H^{\alpha,\beta-k}}$.
\end{proposition}
\begin{proof}
We will prove the result assuming $T$ is small. To prove the general result, it suffices
to split the whole time interval $[0,T]$ into small sub-intervals (whose size depends only
on the bounds (\ref{eq:gammabd}) and (\ref{eq:zetabd})) and iterate the same argument as
many times as needed.\footnote{Note in particular that
$\left\Vert B(\gamma(t),\gamma(t)) \right\Vert_{L^1_T H^{\alpha,\beta}}$ scales at worst like
$T^{\frac{1}{2}}$ for $T$ small under the hypotheses of the proposition, so this norm will
certainly be small if $T$ is chosen small. By a time translation argument the same control
holds on any small interval $[t_0,t_0+T]$ for as long as $\gamma(t)$ remains bounded
in $H^{\alpha,\beta}$.}

Let us denote by $\Delta_{\pm}^{(2)} = \sum_{i=1}^2 \left(
\Delta_{x_i} - \Delta_{x_i^\prime}\right)$ the Laplace operator acting on two particles.
Using Lemma \ref{lemma:gkp-eq}, we easily obtain:
\begin{equation}
\begin{aligned}
& \left( i \partial_t + \frac{1}{2} \Delta_{\pm}^{(2)} \right) \left(
\gamma_{k,+} \otimes \gamma \right) =
\gamma_{k,+} \otimes B(\gamma,\gamma) +
B(\gamma_{k,+},\gamma) \otimes \gamma + \\
& \qquad \qquad \qquad \qquad \qquad\qquad
+ \left( k \frac{x+x^\prime}{\left< x+x^\prime \right>} \cdot
\left( \nabla_x - \nabla_{x^\prime}\right) \gamma_{k-1,+} \right) \otimes \gamma
\end{aligned}
\end{equation}
In integral form, this is:
\begin{equation}
\label{eq:duhamel1}
\begin{aligned}
& \left( \gamma_{k,+} \otimes \gamma \right) (t) =
e^{\frac{1}{2} i t \Delta_{\pm}^{(2)}} \left( \gamma_{k,+} \otimes \gamma \right) (0) \\
& \qquad \quad - i \int_0^t e^{\frac{1}{2} i(t-t_1) \Delta_{\pm}^{(2)}} 
\left( \gamma_{k,+} \otimes B(\gamma,\gamma) \right)(t_1) dt_1 \\
& \qquad \quad
- i \int_0^t e^{\frac{1}{2} i(t-t_1) \Delta_{\pm}^{(2)}} \left( B (\gamma_{k,+},\gamma)
\otimes \gamma\right) (t_1) dt_1 \\
& \qquad \quad - i \int_0^t e^{\frac{1}{2} i(t-t_1) \Delta_{\pm}^{(2)}}
\left\{ \left( k \frac{x+x^\prime}{\left< x+x^\prime \right>} \cdot
\left( \nabla_x - \nabla_{x^\prime} \right) \gamma_{k-1,+} \right) \otimes \gamma 
\right\} (t_1) dt_1.
\end{aligned} 
\end{equation}
The key step is to apply the collision integral to each side of (\ref{eq:duhamel1})
to obtain a (nearly) closed equation for $\zeta_{k,+}$; this idea is adapted from
\cite{PC2010}. To fully close the system we need to incorporate the equation for
$\gamma_{k,+} (t)$ which comes directly by re-writing Lemma \ref{lemma:gkp-eq} in
integral form. Altogether we need to solve the following \emph{system} of equations,
where $\Delta_{\pm} = \Delta_x - \Delta_{x^\prime}$:
\begin{equation}
\label{eq:duhamel2}
\begin{aligned}
& \gamma_{k,+} (t) = e^{\frac{1}{2} i t \Delta_{\pm}} \gamma_{k,+} (0) 
- i \int_0^t e^{\frac{1}{2} i (t-t_1) \Delta_{\pm}} \zeta_{k,+} (t_1) dt_1 \\
& \qquad \qquad \qquad - i \int_0^t e^{\frac{1}{2} i(t-t_1) \Delta_{\pm}}
\left( k \frac{x+x^\prime}{\left< x+x^\prime \right>} \cdot
\left( \nabla_{x} - \nabla_{x^\prime} \right) \gamma_{k-1,+} (t_1) \right) dt_1
\end{aligned}
\end{equation}
\begin{equation}
\label{eq:duhamel3}
\begin{aligned}
& \zeta_{k,+} (t) =
B \left( e^{\frac{1}{2} i t \Delta_{\pm}} \gamma_{k,+} (0), 
e^{\frac{1}{2} i t \Delta_{\pm}} \gamma(0) \right) \\
& - i \int_0^t B \left( e^{\frac{1}{2} i (t-t_1) \Delta_{\pm}} \gamma_{k,+} (t_1),
e^{\frac{1}{2} i(t-t_1) \Delta_{\pm}} \zeta(t_1) \right) dt_1 \\
& - i \int_0^t B \left( e^{\frac{1}{2} i(t-t_1) \Delta_{\pm}} \zeta_{k,+} (t_1), 
e^{\frac{1}{2} i(t-t_1) \Delta_{\pm}} \gamma (t_1) \right) dt_1 \\
& - i \int_0^t B \left(
\begin{aligned}
& e^{\frac{1}{2} i(t-t_1) \Delta_{\pm}} \left\{
k \frac{x+x^\prime}{\left< x+x^\prime \right>} \cdot 
\left( \nabla_x - \nabla_{x^\prime}\right) \gamma_{k-1,+} (t_1)\right\},\\
& \qquad \qquad \qquad \qquad \qquad \qquad \qquad \qquad \qquad
 e^{\frac{1}{2} i(t-t_1) \Delta_{\pm}} \gamma(t_1)
\end{aligned}
\right) dt_1.
\end{aligned}
\end{equation}
We can solve \eqref{eq:duhamel2}-\eqref{eq:duhamel3} on $[0,T]$ for sufficiently small
$T$ by applying a Picard iteration using
Proposition \ref{prop:KeyProp}; we omit the details. In any case the only fact to
be drawn from the fixed point iteration is that the moments
$\gamma_{k,+}$ do not instantaneously diverge in $H^{\alpha,\beta-k}$, so the 
\emph{quantitative} estimates we prove next are justified.

First, using the fact that the propagator $e^{\frac{1}{2} i t \Delta_{\pm}}$ preserves the
spaces $H^{\alpha,\beta}$, we easily obtain from (\ref{eq:duhamel2}) the following
bound:
\begin{equation}
\label{eq:bd1}
\begin{aligned}
& \left\Vert \gamma_{k,+} \right\Vert_{L^\infty_T H^{\alpha,\beta-k}} \leq \\
& \qquad \leq \left\Vert \gamma_{k,+} (0) \right\Vert_{H^{\alpha,\beta-k}}
+ \left\Vert \zeta_{k,+} \right\Vert_{L^1_T H^{\alpha,\beta-k}}
+ C_{\alpha,\beta} T 
\left\Vert \gamma_{k-1,+} \right\Vert_{L^\infty_T H^{\alpha,\beta-(k-1)}}.
\end{aligned}
\end{equation}
For the second estimate, we take the $L^1_T H^{\alpha,\beta-k}$ norm on both sides
of (\ref{eq:duhamel3}). We obtain:
\begin{equation}
\begin{aligned}
& \left\Vert \zeta_{k,+} (t) \right\Vert_{L^1_T H^{\alpha,\beta-k}} \leq
\left\Vert 
B \left( e^{\frac{1}{2} i t \Delta_{\pm}} \gamma_{k,+} (0), 
e^{\frac{1}{2} i t \Delta_{\pm}} \gamma(0) \right) 
\right\Vert_{L^1_T H^{\alpha,\beta-k}} \\
& + \int_0^T \int_0^t \left\Vert 
B \left( e^{\frac{1}{2} i (t-t_1) \Delta_{\pm}} \gamma_{k,+} (t_1),
e^{\frac{1}{2} i(t-t_1) \Delta_{\pm}} \zeta(t_1) \right)
\right\Vert_{H^{\alpha,\beta-k}} dt_1 dt \\
& + \int_0^T \int_0^t \left\Vert
B \left( e^{\frac{1}{2} i(t-t_1) \Delta_{\pm}} \zeta_{k,+} (t_1), 
e^{\frac{1}{2} i(t-t_1) \Delta_{\pm}} \gamma (t_1) \right) 
\right\Vert_{H^{\alpha,\beta-k}} dt_1 dt\\
& + \int_0^T \int_0^t dt_1 dt \times \\
& \qquad \times  \left\Vert B \left(
\begin{aligned}
& e^{\frac{1}{2} i(t-t_1) \Delta_{\pm}} \left\{
k \frac{x+x^\prime}{\left< x+x^\prime \right>} \cdot 
\left( \nabla_x - \nabla_{x^\prime}\right) \gamma_{k-1,+} (t_1)\right\},\\
& \qquad \qquad \qquad \qquad \qquad \qquad \qquad \quad
 e^{\frac{1}{2} i(t-t_1) \Delta_{\pm}} \gamma(t_1)
\end{aligned}
\right) \right\Vert_{H^{\alpha,\beta-k}}.
\end{aligned}
\end{equation}
Now we bound  $\int_0^t dt_1 (\dots) $ by $\int_0^T dt_1 (\dots) $ and apply Fubini.
\begin{equation}
\begin{aligned}
\begin{aligned}
& \left\Vert \zeta_{k,+} (t) \right\Vert_{L^1_T H^{\alpha,\beta-k}} \leq
\left\Vert 
B \left( e^{\frac{1}{2} i t \Delta_{\pm}} \gamma_{k,+} (0),
 e^{\frac{1}{2} i t \Delta_{\pm}} \gamma(0) \right) 
\right\Vert_{L^1_T H^{\alpha,\beta-k}} \\
& + \int_0^T \left\Vert 
B \left( e^{\frac{1}{2} i (t-t_1) \Delta_{\pm}} \gamma_{k,+} (t_1),
e^{\frac{1}{2} i(t-t_1) \Delta_{\pm}} \zeta(t_1) \right)
\right\Vert_{L^1_T H^{\alpha,\beta-k}} dt_1 \\
& + \int_0^T  \left\Vert
B \left( e^{\frac{1}{2} i(t-t_1) \Delta_{\pm}} \zeta_{k,+} (t_1), 
e^{\frac{1}{2} i(t-t_1) \Delta_{\pm}} \gamma (t_1) \right) 
\right\Vert_{L^1_T H^{\alpha,\beta-k}} dt_1\\
& + \int_0^T dt_1  \times \\
& \qquad \times  \left\Vert B \left(
\begin{aligned}
& e^{\frac{1}{2} i(t-t_1) \Delta_{\pm}} \left\{
k \frac{x+x^\prime}{\left< x+x^\prime \right>} \cdot 
\left( \nabla_x - \nabla_{x^\prime}\right) \gamma_{k-1,+} (t_1)\right\},\\
& \qquad \qquad \qquad \qquad \qquad \qquad \qquad 
 e^{\frac{1}{2} i(t-t_1) \Delta_{\pm}} \gamma(t_1)
\end{aligned}
\right) \right\Vert_{L^1_T H^{\alpha,\beta-k}}.
\end{aligned}
\end{aligned}
\end{equation}
Finally we apply Cauchy-Schwarz to bound $\left\Vert \dots \right\Vert_{L^1_T}$ by
$T^{\frac{1}{2}} \left\Vert \dots \right\Vert_{L^2_T}$; then, we are able to apply
Proposition \ref{prop:KeyProp} to deduce the following bound:
\begin{equation}
\begin{aligned}
& \left\Vert \zeta_{k,+} \right\Vert_{L^1_T H^{\alpha,\beta-k}} \leq
C T^{\frac{1}{2}} \left\Vert \gamma_{k,+} (0) \right\Vert_{H^{\alpha,\beta-k}}
\left\Vert \gamma(0) \right\Vert_{H^{\alpha,\beta-k}} + \\
& +
C T^{\frac{1}{2}} \left\Vert \gamma_{k,+} \right\Vert_{L^\infty_T H^{\alpha,\beta-k}}
\left\Vert \zeta \right\Vert_{L^1_T H^{\alpha,\beta-k}} +
C T^{\frac{1}{2}} \left\Vert \zeta_{k,+} \right\Vert_{L^1_T H^{\alpha,\beta-k}}
\left\Vert \gamma \right\Vert_{L^\infty_T H^{\alpha,\beta-k}} + \\
& + C_{\alpha,\beta} k T^{\frac{3}{2}} 
\left\Vert \gamma_{k-1,+} \right\Vert_{L^\infty_T H^{\alpha,\beta-(k-1)}}
\left\Vert \gamma \right\Vert_{L^\infty_T H^{\alpha,\beta-k}}.
\end{aligned}
\end{equation}
Since $H^{\alpha,\beta} \subset H^{\alpha,\beta-k}$, this implies:
\begin{equation}
\label{eq:bd2}
\begin{aligned}
& \left\Vert \zeta_{k,+} \right\Vert_{L^1_T H^{\alpha,\beta-k}} \leq
C T^{\frac{1}{2}} \left\Vert \gamma_{k,+} (0) \right\Vert_{H^{\alpha,\beta-k}}
\left\Vert \gamma(0) \right\Vert_{H^{\alpha,\beta}} + \\
& +
C T^{\frac{1}{2}} \left\Vert \gamma_{k,+} \right\Vert_{L^\infty_T H^{\alpha,\beta-k}}
\left\Vert \zeta \right\Vert_{L^1_T H^{\alpha,\beta}} +
C T^{\frac{1}{2}} \left\Vert \zeta_{k,+} \right\Vert_{L^1_T H^{\alpha,\beta-k}}
\left\Vert \gamma \right\Vert_{L^\infty_T H^{\alpha,\beta}} + \\
& + C_{\alpha,\beta} k T^{\frac{3}{2}} 
\left\Vert \gamma_{k-1,+} \right\Vert_{L^\infty_T H^{\alpha,\beta-(k-1)}}
\left\Vert \gamma \right\Vert_{L^\infty_T H^{\alpha,\beta}}.
\end{aligned}
\end{equation}

Let us define
\begin{equation}
M_T = T^{\frac{1}{2}} \left\Vert \gamma_{k,+} \right\Vert_{L^\infty_T H^{\alpha,\beta-k}}
+ \left\Vert \zeta_{k,+} \right\Vert_{L^1_T H^{\alpha,\beta}}.
\end{equation}
Then combining (\ref{eq:bd1}) and (\ref{eq:bd2}), we obtain:
\begin{equation}
\begin{aligned}
& M_T \leq C \left( T^{\frac{1}{2}} + T^{\frac{1}{2}} \left\Vert \gamma 
\right\Vert_{L^\infty_T H^{\alpha,\beta}} +
\left\Vert \zeta \right\Vert_{L^1_T H^{\alpha,\beta}} \right) M_T + \\
& \qquad \qquad + T^{\frac{1}{2}} \left\Vert \gamma_{k,+} (0) \right\Vert_{H^{\alpha,\beta-k}}
+ T^{\frac{1}{2}} \left\Vert \gamma_{k,+} (0) \right\Vert_{H^{\alpha,\beta-k}}
\left\Vert \gamma(0) \right\Vert_{H^{\alpha,\beta}} + \\
& \qquad \qquad + C_{\alpha,\beta} T^{\frac{3}{2}}
\left\Vert \gamma_{k-1,+} \right\Vert_{L^\infty_T H^{\alpha,\beta-(k-1)}} + \\
& \qquad \qquad  + C_{\alpha,\beta} k T^{\frac{3}{2}}
\left\Vert \gamma_{k-1,+} \right\Vert_{L^\infty_T H^{\alpha,\beta-(k-1)}}
\left\Vert \gamma \right\Vert_{L^\infty_T H^{\alpha,\beta}}.
\end{aligned}
\end{equation}
Since $\left\Vert \zeta \right\Vert_{L^1_T H^{\alpha,\beta}}$ is
$\mathcal{O} (T^{\frac{1}{2}})$ for small $T$, we find that the prefactor of $M_T$ on
the right-hand side is small if $T$ is small. The smallness of $T$ depends only on
the underlying solution $\gamma (t)$ of Boltzmann's equation.
\end{proof}

\subsection{Moments in $x-x^\prime$.} In this subsection we propagate moments of $\gamma$
in $x-x^\prime$, which is equivalent to propagating derivatives of the distribution
$f(x,v)$ in the velocity variable. As in the previous subsection, we will write an
equation for the $(2k)$th moment of $\gamma$ and use the existence of a solution 
$\gamma(t)$ in $H^{\alpha,\beta}$ for large enough $\alpha$. However, as we will see,
the proof is much more technical both because (\ref{eq:Mts-Bp-m}) introduces many
new terms and because we can only close the estimate for moments of \emph{even} order.

\begin{lemma}
\label{lemma:gkm-eq}
Consider a distributional solution $\gamma (t)$, $t\in [0,T]$, of the Boltzmann equation.
Then for any $k\in \mathbb{N}$, $k\geq 1$, there holds
\begin{equation}
\label{eq:gkm-1}
\begin{aligned}
& \left( i \partial_t + \frac{1}{2} \left( \Delta_x - \Delta_{x^\prime} \right) \right)
\left( \left< x-x^\prime \right>^{2k} \gamma(t) \right) =\\
& = 
B \left( \left< x-x^\prime \right>^{2k} \gamma(t), \gamma(t) \right) 
+ B^+ \left( \gamma (t), \left< x-x^\prime \right>^{2k} \gamma(t) \right) + \\
& + \sum_{\substack{ j_1+j_2+j_3 = k \\
j_1 \neq k \\ j_2 \neq k}} \binom{k}{j_1,j_2,j_3} (-1)^{j_3} 
B^+ \left( \left< x-x^\prime \right>^{2j_1} \gamma(t),
\left< x-x^\prime \right>^{2j_2} \gamma(t) \right) +\\
& + 2k \left( \nabla_x + \nabla_{x^\prime} \right) \cdot
\left( (x-x^\prime) \left< x-x^\prime \right>^{2k-2} \gamma(t) \right)
\end{aligned}
\end{equation}
Additionally, for any $k\in \mathbb{N}$, $k\geq 1$, there holds
\begin{equation}
\label{eq:gkm-2}
\begin{aligned}
& \left( i \partial_t + \frac{1}{2} \left( \Delta_x - \Delta_{x^\prime} \right)\right)
\left( (x-x^\prime) \left< x-x^\prime \right>^{2k-2} \gamma(t) \right)
= \\
& = B \left( (x-x^\prime) \left< x-x^\prime \right>^{2k-2} \gamma(t),\gamma(t) \right) + \\
& + B^+ \left( 
\gamma (t), (x-x^\prime) \left< x-x^\prime \right>^{2k-2} \gamma(t) \right) + \\
& + \sum_{\substack{ j_1+j_2+j_3=k-1 \\ j_1 \neq k-1}} \binom{k-1}{j_1,j_2,j_3}
(-1)^{j_3} B^+  \left( 
\begin{aligned}
& (x-x^\prime) \left< x-x^\prime \right>^{2j_1} \gamma (t), \\
&\qquad \qquad \qquad  \left< x-x^\prime \right>^{2j_2} \gamma (t)
\end{aligned}
 \right) + \\
& + \sum_{\substack{j_1 + j_2 + j_3 = k-1 \\ j_2 \neq k-1}}
\binom{k-1}{j_1,j_2,j_3} (-1)^{j_3}
B^+ \left(
\begin{aligned}
& \left< x-x^\prime \right>^{2j_1} \gamma (t), \\
& \qquad  (x-x^\prime) \left< x-x^\prime \right>^{2j_2} \gamma (t) 
\end{aligned}
\right) + \\
& + \left( \nabla_x + \nabla_{x^\prime} \right) \left(
\left< x-x^\prime \right>^{2k-2} \gamma (t) \right) + \\
& + (2k-2)
\left( \frac{x-x^\prime}{\left<x-x^\prime\right>} \cdot \left( 
\nabla_x + \nabla_{x^\prime} \right) \right)
\left( (x-x^\prime) \left< x-x^\prime \right>^{2k-3} \gamma(t) \right)
\end{aligned}
\end{equation}
\end{lemma}
\begin{proof} 
This computation follows by using  \eqref{eq:Mts-Bm-m}-\eqref{eq:Mts-Bp-m-2}.
\end{proof}
\begin{remark}
Note carefully that $(x-x^\prime) \left< x-x^\prime\right>^{2k-2} \gamma (t)$
is a complex vector field.
\end{remark}

Let us introduce the following notation:
\begin{equation}
\zeta (t) = B (\gamma(t),\gamma (t))
\end{equation}
\begin{equation}
\gamma_{k,-} (t,x,x^\prime) = \left< x-x^\prime\right>^k \gamma (t,x,x^\prime)
\end{equation}
\begin{equation}
\zeta_{k,-} (t) = B \left( \gamma_{k,-} (t),\gamma(t)\right) +
B^+ \left( \gamma(t), \gamma_{k,-} (t) \right)
\end{equation}
\begin{equation}
\breve{\gamma}_{k,-} = (x-x^\prime) \left< x-x^\prime\right>^{k-1} \gamma(t,x,x^\prime)
\end{equation}
\begin{equation}
\breve{\zeta}_{k,-} (t) = B \left( \breve{\gamma}_{k,-} (t), \gamma(t) \right) +
B^+ \left( \gamma(t), \breve{\gamma}_{k,-} (t) \right)
\end{equation}

\begin{proposition}
Let $\gamma (t)$ be a solution of Boltzmann's equation with \textbf{constant} collision
kernel, $\mathbf{b} \equiv \textnormal{cst.}$, satisfying the following bounds on some
time interval $[0,T]$:
\begin{equation}
\left\Vert \gamma(t) \right\Vert_{L^\infty_T H^{\alpha,\beta}} < \infty
\end{equation}
\begin{equation}
\left\Vert B(\gamma(t),\gamma(t) \right\Vert_{L^1_T H^{\alpha,\beta}} < \infty
\end{equation}
with $\alpha,\beta > \frac{d-1}{2}$.
Further assume that, for some integer $K > 0$ with
$2K < \alpha - \frac{d-1}{2}$, for all $1 \leq k \leq K$ there holds
\begin{equation}
\left\Vert \gamma_{2k,-} (0) \right\Vert_{H^{\alpha-2k,\beta}} < \infty
\end{equation}
Then for all $1 \leq k \leq K$ we have
\begin{equation}
\left\Vert \gamma_{2k,-} (t) \right\Vert_{L^\infty_T H^{\alpha-2k,\beta}} < \infty.
\end{equation}
Moreover there is an explicit bound on $\underset{1\leq k \leq K}{\sup}
\left\Vert \gamma_{2k,-} (t) \right\Vert_{L^\infty_T H^{\alpha-2k,\beta}}$ that only
depends on $\left\Vert \gamma(t) \right\Vert_{L^\infty_T H^{\alpha,\beta}}$,
$\left\Vert B(\gamma(t),\gamma(t)) \right\Vert_{L^1_T H^{\alpha,\beta}}$, and
$\underset{1\leq k \leq K}{\sup} \left\Vert
\gamma_{2k,-} (0) \right\Vert_{H^{\alpha-2k,\beta}}$.
\end{proposition}
\begin{proof}
We will prove the result assuming $T$ is small. To prove the general result, it suffices
to split the whole time interval $[0,T]$ into small sub-intervals and iterate the
argument, as in Proposition \ref{prop:gkp-mts}.

Similar to Proposition \ref{prop:gkp-mts}, the main idea is to write a closed equation for
the system $\left\{ \gamma_{2k,-}, \zeta_{2k,-}, \breve{\gamma}_{2k-1,-},
\breve{\zeta}_{2k-1,-} \right\}$. Since the computations are quite involved, we only write
down the main steps. We will \emph{assume} for the induction that, if
$j_1,j_2 \leq k-1$ and $j_1 + j_2 \leq k$, then
\begin{equation}
\label{eq:Assume1}
 B^+ \left( \gamma_{2j_1, -},\gamma_{2j_2,-} \right) \in
L^1_T H^{\alpha-2k,\beta},
\end{equation}
and that if $j_1 \leq k-2$ and $j_1+j_2 \leq k-1$, then
\begin{equation}
\label{eq:Assume2}
B^+ \left( \breve{\gamma}_{2j_1+1,-},\gamma_{2j_2,-} \right) \in
L^1_T H^{\alpha-2k+1,\beta},
\end{equation}
and that if $j_2 \leq k-2$ and $j_1 + j_2 \leq k-1$, then
\begin{equation}
\label{eq:Assume3}
B^+ \left( \gamma_{2j_1,-},\breve{\gamma}_{2j_2+1,-} \right) \in
L^1_T H^{\alpha-2k+1,\beta}.
\end{equation}
These assumptions will, of course, have to be verified later (the case $k=1$ is
easily checked using the facts that $\gamma \in L^\infty_T H^{\alpha,\beta}$
and $\zeta \in L^1_T H^{\alpha,\beta}$).

We observe first that (\ref{eq:gkm-1}) is equivalent to the following system of
equations (this system is \emph{not} closed due to the presence of 
$\breve{\gamma}_{2k-1,-}$, which is not given to us by the inductive hypothesis):
\begin{equation}
\begin{aligned}
& \gamma_{2k,-} (t) = e^{\frac{1}{2} i t \Delta_{\pm}} \gamma_{2k,-} (0)
- i \int_0^t e^{\frac{1}{2} i (t-t_1) \Delta_{\pm}} \zeta_{2k,-} (t_1) dt_1 \\
& - i \sum_{\substack{j_1+j_2+j_3=k \\ j_1 \neq k \\ j_2 \neq k}}
\binom{k}{j_1,j_2,j_3} (-1)^{j_3} \int_0^t
e^{\frac{1}{2} i (t-t_1) \Delta_{\pm}} B^+ \left(
\gamma_{2j_1,-} (t_1), \gamma_{2j_2,-} (t_1) \right) dt_1 \\
& - 2 k i \int_0^t e^{\frac{1}{2} i (t-t_1) \Delta_{\pm}}
\left( \left( \nabla_x + \nabla_{x^\prime}\right) \cdot \breve{\gamma}_{2k-1,-} (t_1)
\right) dt_1.
\end{aligned}
\end{equation}

Also using a Duhamel expression for  $\left( \gamma_{2k,-} \otimes \gamma \right) (t)$, 
which can be obtained in a similar way as  \eqref{eq:duhamel1}, we obtain: 
\begin{equation}
\begin{aligned}
& \zeta_{2k,-} (t) =
B \left( e^{\frac{1}{2} i t \Delta_{\pm}} \gamma_{2k,-} (0),
e^{\frac{1}{2} i t \Delta_{\pm}} \gamma(0) \right)
+ B^+ \left( e^{\frac{1}{2} i t \Delta_{\pm}} \gamma(0),
e^{\frac{1}{2} i t \Delta_{\pm}} \gamma_{2k,-} (0) \right) \\
& - i \int_0^t B \left(
e^{\frac{1}{2} i (t-t_1) \Delta_{\pm}} \gamma_{2k,-} (t_1),
e^{\frac{1}{2} i(t-t_1)\Delta_{\pm}} \zeta(t_1)\right) dt_1 \\
& - i \int_0^t B^+ \left(
e^{\frac{1}{2} i (t-t_1) \Delta_{\pm}} \zeta(t_1),
e^{\frac{1}{2} i(t-t_1)\Delta_{\pm}} \gamma_{2k,-} (t_1) \right) dt_1\\
& - i \int_0^t
B\left( e^{\frac{1}{2} i(t-t_1)\Delta_{\pm}} \zeta_{2k,-} (t_1),
e^{\frac{1}{2} i(t-t_1)\Delta_{\pm}} \gamma(t_1) \right) dt_1\\
& - i \int_0^t B^+ \left( e^{\frac{1}{2} i(t-t_1)\Delta_{\pm}} \gamma(t_1),
e^{\frac{1}{2} i(t-t_1)\Delta_{\pm}} \zeta_{2k,-} (t_1) \right) dt_1 \\
& - i \sum_{\substack{j_1+j_2+j_3=k \\ j_1 \neq k \\ j_2 \neq k}}
\binom{k}{j_1,j_2,j_3} (-1)^{j_3} \times \\
& \qquad \times \int_0^t B\left( 
e^{\frac{1}{2} i (t-t_1) \Delta_{\pm}} B^+ \left( \gamma_{2j_1,-}(t_1),\gamma_{2j_2,-}(t_1)\right),
e^{\frac{1}{2} i (t-t_1) \Delta_{\pm}} \gamma(t_1) \right) dt_1 \\
& - i \sum_{\substack{j_1+j_2+j_3 = k \\ j_1 \neq k \\ j_2 \neq k}}
\binom{k}{j_1,j_2,j_3} (-1)^{j_3} \times \\
& \qquad \times \int_0^t B^+ \left(
e^{\frac{1}{2} i(t-t_1)\Delta_{\pm}} \gamma(t_1),
e^{\frac{1}{2} i(t-t_1)\Delta_{\pm}} B^+ \left(
\gamma_{2j_1,-}(t_1),\gamma_{2j_2,-}(t_1)\right) \right) dt_1 \\
& - 2 k i \int_0^t B \left( e^{\frac{1}{2} i(t-t_1)\Delta_{\pm}} \left(
\left( \nabla_x + \nabla_{x^\prime}\right) \cdot \breve{\gamma}_{2k-1,-} (t_1) \right),
e^{\frac{1}{2} i (t-t_1) \Delta_{\pm}} \gamma (t_1) \right) dt_1 \\
& - 2 k i \int_0^t
B^+ \left( e^{\frac{1}{2} i(t-t_1)\Delta_{\pm}} \gamma(t_1),
e^{\frac{1}{2} i(t-t_1)\Delta_{\pm}} \left( \left( \nabla_x + \nabla_{x^\prime}\right)
\cdot \breve{\gamma}_{2k-1,-}(t_1)\right)\right) dt_.1
\end{aligned}
\end{equation}
Arguing as in Proposition \ref{prop:gkp-mts}, and applying 
Proposition \ref{prop:KeyProp}, we deduce the following estimates:
\begin{equation}
\label{eq:bd10}
\begin{aligned}
& \left\Vert \gamma_{2k,-} \right\Vert_{L^\infty_T H^{\alpha-2k,\beta}} \leq
\left\Vert \gamma_{2k,-} (0) \right\Vert_{H^{\alpha-2k,\beta}} +
\left\Vert \zeta_{2k,-} \right\Vert_{L^1_T H^{\alpha-2k,\beta}} + \\
& \qquad \qquad + C_k \sup_{\substack{j_1+j_2 \leq k \\ j_1 \neq k \\ j_2 \neq k}}
\left\Vert B^+ \left( \gamma_{2j_1,-} , \gamma_{2j_2,-}\right)
\right\Vert_{L^1_T H^{\alpha-2k,\beta}}  + \\
& \qquad \qquad
+ 2 k T \left\Vert \breve{\gamma}_{2k-1,-} \right\Vert_{L^\infty_T H^{\alpha-(2k-1),\beta}},
\end{aligned}
\end{equation}
\begin{equation}
\label{eq:bd11}
\begin{aligned}
& \left\Vert \zeta_{2k,-} \right\Vert_{L^1_T H^{\alpha-2k,\beta}} \leq
C T^{\frac{1}{2}} \left\Vert \gamma_{2k,-} (0) \right\Vert_{H^{\alpha-2k,\beta}}
\left\Vert \gamma \right\Vert_{L^\infty_T H^{\alpha,\beta}} +\\
& \qquad \qquad
+ C T^{\frac{1}{2}} \left\Vert \gamma_{2k,-} \right\Vert_{L^\infty_T H^{\alpha-2k,\beta}}
\left\Vert \zeta \right\Vert_{L^1_T H^{\alpha,\beta}} + \\
& \qquad \qquad
+ C T^{\frac{1}{2}} \left\Vert \zeta_{2k,-} \right\Vert_{L^1_T H^{\alpha-2k,\beta}}
\left\Vert \gamma \right\Vert_{L^\infty_T H^{\alpha,\beta}} + \\
& \qquad \qquad
+ C_k T^{\frac{1}{2}} \left\Vert \gamma \right\Vert_{L^\infty_T H^{\alpha,\beta}}
\sup_{\substack{j_1+j_2 \leq k \\ j_1 \neq k \\ j_2 \neq k}} 
\left\Vert B^+ \left( \gamma_{2j_1,-} , \gamma_{2j_2,-} \right)
\right\Vert_{L^1_T H^{\alpha-2k,\beta}} + \\
& \qquad \qquad
+ C k T^{\frac{3}{2}} \left\Vert \gamma \right\Vert_{L^\infty_T H^{\alpha,\beta}}
\left\Vert \breve{\gamma}_{2k-1,-} \right\Vert_{L^\infty_T H^{\alpha-(2k-1),\beta}}.
\end{aligned}
\end{equation}
Observe that \eqref{eq:bd10}-\eqref{eq:bd11} does \emph{not} yield a closed estimate
in terms of $\gamma_{j,-}$ (with $j\leq 2k-2$) precisely because of the terms involving
$\breve{\gamma}_{2k-1,-}$. This is why we have to solve (\ref{eq:gkm-2})
simultaneously with (\ref{eq:gkm-1}).

Obviously the equation (\ref{eq:gkm-2}) yields a system of equations for the pair
$\left\{ \breve{\gamma}_{2k-1,-}, \breve{\zeta}_{2k-1,-}\right\}$, but this system is
very cumbersome to write down. Instead, we will simply note the resulting estimates:
\begin{equation}
\label{eq:bd20}
\begin{aligned}
& \left\Vert \breve{\gamma}_{2k-1,-} \right\Vert_{L^\infty_T H^{\alpha-(2k-1),\beta}} 
\leq \left\Vert \breve{\gamma}_{2k-1,-} (0) \right\Vert_{H^{\alpha-(2k-1),\beta}} + \\
&  \qquad \qquad
+ \left\Vert \breve{\zeta}_{2k-1,-} \right\Vert_{L^1_T H^{\alpha-(2k-1),\beta}} + \\
& \qquad \qquad + C_k \sup_{\substack{j_1+j_2 \leq k-1 \\ j_1 \neq k-1}}
\left\Vert B^+ \left( \breve{\gamma}_{2j_1+1,-}, \gamma_{2j_2,-} \right)
\right\Vert_{L^1_T H^{\alpha-(2k-1),\beta}} + \\
& \qquad \qquad + C_k \sup_{\substack{j_1+j_2 \leq k-1 \\ j_2 \neq k-1}}
\left\Vert B^+ \left( \gamma_{2j_1,-},\breve{\gamma}_{2j_2+1,-} \right)
\right\Vert_{L^1_T H^{\alpha-(2k-1),\beta}} + \\
& \qquad \qquad
+ C T \left\Vert \gamma_{2k-2,-} \right\Vert_{L^\infty_T H^{\alpha-(2k-2),\beta}} + \\
& \qquad \qquad + C_{\alpha,\beta} (2k-2) T \left\Vert
\breve{\gamma}_{2k-2,-} \right\Vert_{L^\infty_T H^{\alpha-(2k-2),\beta}}.
\end{aligned}
\end{equation}
\begin{equation}
\label{eq:bd21}
\begin{aligned}
& \left\Vert \breve{\zeta}_{2k-1,-} \right\Vert_{L^1_T H^{\alpha-(2k-1),\beta}}
\leq C T^{\frac{1}{2}} \left\Vert \breve{\gamma}_{2k-1,-} (0) 
\right\Vert_{H^{\alpha-(2k-1),\beta}}
\left\Vert \gamma \right\Vert_{L^\infty_T H^{\alpha,\beta}} + \\
& \qquad \quad + C T^{\frac{1}{2}}
\left\Vert \breve{\gamma}_{2k-1,-} \right\Vert_{L^\infty_T H^{\alpha-(2k-1),\beta}}
\left\Vert \zeta \right\Vert_{L^1_T H^{\alpha,\beta}} + \\
& \qquad \quad + C T^{\frac{1}{2}}
\left\Vert \breve{\zeta}_{2k-1,-} \right\Vert_{L^1_T H^{\alpha-(2k-1),\beta}}
\left\Vert \gamma \right\Vert_{L^\infty_T H^{\alpha,\beta}} + \\
& \qquad \quad
+ C_k T^{\frac{1}{2}} \left\Vert \gamma \right\Vert_{L^\infty_T H^{\alpha,\beta}}
\sup_{\substack{j_1+j_2 \leq k-1 \\ j_1 \neq k-1}}
\left\Vert B^+ \left( \breve{\gamma}_{2j_1+1,-}, \gamma_{2j_2,-} \right)
\right\Vert_{L^1_T H^{\alpha-(2k-1),\beta}} + \\
& \qquad \quad
+ C_k T^{\frac{1}{2}} \left\Vert \gamma \right\Vert_{L^\infty_T H^{\alpha,\beta}}
\sup_{\substack{j_1+j_2 \leq k-1 \\ j_2 \neq k-1}}
\left\Vert B^+ \left( \gamma_{2j_1,-},\breve{\gamma}_{2j_2+1,-}
\right) \right\Vert_{L^1_T H^{\alpha-(2k-1),\beta}} + \\
& \qquad \quad
+ C T^{\frac{3}{2}} \left\Vert \gamma \right\Vert_{L^\infty_T H^{\alpha,\beta}}
\left\Vert \gamma_{2k-2,-} \right\Vert_{L^\infty_T H^{\alpha-(2k-2),\beta}} + \\
& \qquad \quad + C_{\alpha,\beta} (2k-2) T^{\frac{3}{2}}
\left\Vert \gamma \right\Vert_{L^\infty_T H^{\alpha,\beta}}
\left\Vert \breve{\gamma}_{2k-2,-} 
\right\Vert_{L^\infty_T H^{\alpha-(2k-2),\beta}}.
\end{aligned}
\end{equation}
Combining \eqref{eq:bd20}-\eqref{eq:bd21} and using \eqref{eq:Assume2}-\eqref{eq:Assume3},
we can conclude the for sufficiently small $T$ depending only on 
$\left\Vert \gamma \right\Vert_{L^\infty_T H^{\alpha,\beta}}$ and
$\left\Vert B(\gamma,\gamma) \right\Vert_{L^1_T H^{\alpha,\beta}}$, we have
\begin{equation}
\label{eq:bd30}
\breve{\gamma}_{2k-1,-} \in L^\infty_T H^{\alpha-(2k-1),\beta}
\end{equation}
\begin{equation}
\label{eq:bd31}
\breve{\zeta}_{2k-1,-} \in L^1_T H^{\alpha-(2k-1),\beta}.
\end{equation}
Now we can combine \eqref{eq:bd10}-\eqref{eq:bd11} with \eqref{eq:bd30} and
\eqref{eq:Assume1} to conclude that
\begin{equation}
\label{eq:bd40}
\gamma_{2k,-} \in L^\infty_T H^{\alpha-2k,\beta}
\end{equation}
\begin{equation}
\label{eq:bd41}
\zeta_{2k,-} \in L^1_T H^{\alpha-2k,\beta}.
\end{equation}

Finally, we must verify the assumptions (\ref{eq:Assume1}-\ref{eq:Assume3}) which were
used in the inductive process. The point here is that it is not enough to prove
that $\gamma_{2k,-} \in L^\infty_T H^{\alpha-2k,\beta}$ and
$\breve{\gamma}_{2k-1,-} \in L^\infty_T H^{\alpha-(2k-1),\beta}$, because we do not
have continuity bounds for the operator $B^+$ itself. Rather it is essential to use the
facts that $\zeta_{2k,-} \in L^1_T H^{\alpha-2k,\beta}$ and
$\breve{\zeta}_{2k-1,-} \in L^1_T H^{\alpha-(2k-1),\beta}$; fortunately, these bounds are
provided to us by the induction itself, so we may conclude.
\end{proof}

\section{Persistence of Regularity}
\label{sec:pers}

Now the question is as follows: suppose we have a solution $\gamma(t)$ in
$H^{\alpha,\beta}$ on a maximal time interval $[0,T)$, and suppose further
that $\gamma (0) \in H^{\alpha_1,\beta_1}$ for some $\alpha_1 \geq \alpha$
and $\beta_1 \geq \beta$. Then there is a maximal solution $\gamma_1(t)$ in
$H^{\alpha_1,\beta_1}$ which exists on a time interval $[0,T_1)$ with
$\gamma_1 (0) = \gamma (0)$. Clearly
$\gamma_1$ coincides with $\gamma$ on $[0,T_1)$, and in particular
$T_1 \leq T$. Can we say that $T_1 = T$? The results in this section answer this question in
the affirmative when $\mathbf{b}$ is bounded
 and $(\alpha_1 - \alpha) $ is an integer.

\begin{proposition}
Let $\gamma (t)$ be a solution of Boltzmann's equation with
$\left\Vert \mathbf{b} \right\Vert_{\infty} < \infty$, and suppose
$\gamma \in L^\infty_T H^{\alpha,\beta}$ and
$B(\gamma,\gamma) \in L^1_T H^{\alpha,\beta}$ for some
$\alpha,\beta > \frac{d-1}{2}$, and further suppose that
$\gamma (0) \in H^{\alpha+r,\beta}$ for some $r \in \mathbb{N}$.
Then $\gamma \in L^\infty_T H^{\alpha+r,\beta}$
and $B(\gamma,\gamma) \in L^1_T H^{\alpha+r,\beta}$.
\end{proposition}

\begin{proof}
To begin, notice that we have local well-posedness in $H^{\alpha+r,\beta}$,
and $\gamma(0) \in H^{\alpha+r,\beta}$; this suffices to justify our formal
computations. All we have to show is that $\gamma (t)$ remains bounded in
$H^{\alpha+r,\beta}$ on the full time interval $[0,T]$. By an iteration in
time, it suffices to prove the result for $T$ small enough depending only
on $\left\Vert \gamma \right\Vert_{L^\infty_T H^{\alpha,\beta}}$ and
$\left\Vert B(\gamma,\gamma) \right\Vert_{L^1_T H^{\alpha,\beta}}$.

The proof follows by a simple induction using Lemma \ref{lemma:deriv-commute},
combined with Proposition \ref{prop:KeyProp} and the fact that
$\left( \nabla_x + \nabla_{x^\prime}\right)$ commutes with the free
propagator $e^{i t \Delta_{\pm}}$. For notational convenience, we denote by
$\partial_*^{\mathbf{k}}$ a multi-index of the following form:
\begin{equation}
\left( \partial_{x_1} + \partial_{x_1^\prime} \right)^{k_1} 
\left( \partial_{x_2} + \partial_{x_2^\prime} \right)^{k_2} 
\dots \;
\left( \partial_{x_d} + \partial_{x_d^\prime} \right)^{k_d}
\end{equation}
with $k_1+k_2+\dots+k_d = |\mathbf{k}|$. Differentiating Boltzmann's equation and
using Lemma \ref{lemma:deriv-commute}, we have
\begin{equation}
\label{eq:multi-deriv}
\begin{aligned}
& \left( i \partial_t + \frac{1}{2} \left( \Delta_x - \Delta_{x^\prime} \right) \right)
\left( \partial_*^{\mathbf{k}} \gamma \right) =
B \left( \partial_*^{\mathbf{k}} \gamma, \gamma \right) +
B \left( \gamma, \partial_*^{\mathbf{k}} \gamma \right) + \\
& \qquad \qquad \qquad \qquad \qquad \qquad \qquad
+ \sum_{\substack{\mathbf{c} \leq \mathbf{k} \\ \mathbf{c} \neq 0 \\ 
\mathbf{c} \neq \mathbf{k}}}
\frac{\mathbf{k} !}{\mathbf{c}! (\mathbf{k}-\mathbf{c})!}
B \left( \partial_*^{\mathbf{c}} \gamma, \partial_*^{\mathbf{k}-\mathbf{c}} \gamma \right)
\end{aligned}
\end{equation}
where $\mathbf{k}! = k_1! k_2! \dots k_d!$. We assume for the induction that if
$\mathbf{c} \leq \mathbf{k}$, $\mathbf{c} \neq 0$, and $\mathbf{c} \neq \mathbf{k}$,
 then
\begin{equation}
\label{eq:bd55}
B \left( \partial_*^{\mathbf{c}} \gamma, \partial_*^{\mathbf{k}-\mathbf{c}}\gamma \right)
\in L^1_T H^{\alpha,\beta}
\end{equation}
This assertion will eventually be justified by the induction (note that the
base case $|\mathbf{k}| = 1$ is trivial).

Let us define
\begin{equation}
\zeta_{\mathbf{k}} (t) = 
B \left( \partial_*^\mathbf{k} \gamma (t), \gamma (t)\right) +
B \left( \gamma (t), \partial_*^{\mathbf{k}} \gamma (t) \right)
\end{equation}
Then (\ref{eq:multi-deriv}) is equivalent to the following system of equations:
\begin{equation}
\begin{aligned}
& \partial_*^{\mathbf{k}} \gamma (t) =
e^{\frac{1}{2} i t \Delta_{\pm}} \partial_*^{\mathbf{k}} \gamma (0) - i \int_0^T
e^{\frac{1}{2} i (t-t_1) \Delta_{\pm}} \zeta_{\mathbf{k}} (t_1) dt_1 \\
& \qquad \qquad 
- i \sum_{\substack{\mathbf{c} \leq \mathbf{k} \\ \mathbf{c} \neq 0 \\
\mathbf{c} \neq \mathbf{k}}} \frac{\mathbf{k}!}{\mathbf{c}! (\mathbf{k}-\mathbf{c})!}
\int_0^T e^{\frac{1}{2} i (t-t_1) \Delta_{\pm}} B \left( 
\partial_*^{\mathbf{c}} \gamma (t_1),
\partial_*^{\mathbf{k}-\mathbf{c}} \gamma (t_1) \right) dt_1
\end{aligned}
\end{equation}
\begin{equation}
\begin{aligned}
& \zeta_{\mathbf{k}} (t) = \\
& = B \left( e^{\frac{1}{2} i t \Delta_{\pm}} \gamma(0),
e^{\frac{1}{2} i t \Delta_{\pm}} \partial_*^{\mathbf{k}}
\gamma (0) \right)
+  B \left( e^{\frac{1}{2} i (t-t_1) \Delta_{\pm}} \partial_*^{\mathbf{k}}
\gamma(0),e^{\frac{1}{2} i (t-t_1) \Delta_{\pm}} \gamma(0) \right) \\
& - i \int_0^t B\left( e^{\frac{1}{2} i (t-t_1) \Delta_{\pm}} \zeta(t_1),
e^{\frac{1}{2} i (t-t_1) \Delta_{\pm}} \partial_*^{\mathbf{k}} \gamma(t_1) \right) dt_1 \\
& - i \int_0^t B\left( e^{\frac{1}{2} i (t-t_1) \Delta_{\pm}} \partial_*^{\mathbf{k}} \gamma (t_1),
e^{\frac{1}{2} i (t-t_1) \Delta_{\pm}} \zeta (t_1) \right) dt_1 \\
& - i \int_0^t B\left( e^{\frac{1}{2} i (t-t_1) \Delta_{\pm}} \gamma (t_1),
e^{\frac{1}{2} i (t-t_1) \Delta_{\pm}} \zeta_{\mathbf{k}} (t_1) \right) dt_1 \\
& - i \int_0^t B\left( e^{\frac{1}{2} i (t-t_1) \Delta_{\pm}} \zeta_{\mathbf{k}} (t_1),
e^{\frac{1}{2} i (t-t_1) \Delta_{\pm}} \gamma (t_1) \right) dt_1 \\
& - i \sum_{\substack{\mathbf{c}\leq \mathbf{k} \\
\mathbf{c} \neq 0 \\ \mathbf{c} \neq \mathbf{k}}}
\frac{\mathbf{k}!}{\mathbf{c}! (\mathbf{k}-\mathbf{c})!} \times  \\
& \qquad \qquad \times
\int_0^t B\left( e^{\frac{1}{2} i (t-t_1) \Delta_{\pm}} \gamma (t_1),
e^{\frac{1}{2} i (t-t_1) \Delta_{\pm}} 
B\left( \partial_*^{\mathbf{c}} \gamma(t_1),
\partial_*^{\mathbf{k}-\mathbf{c}} \gamma(t_1) \right) \right) dt_1 \\
& - i \sum_{\substack{\mathbf{c}\leq \mathbf{k} \\
\mathbf{c} \neq 0 \\ \mathbf{c} \neq \mathbf{k}}}
\frac{\mathbf{k}!}{\mathbf{c}! (\mathbf{k}-\mathbf{c})!} \times  \\
& \qquad \qquad \times
\int_0^t B\left( e^{\frac{1}{2} i (t-t_1) \Delta_{\pm}} 
B\left( \partial_*^{\mathbf{c}} \gamma(t_1),
\partial_*^{\mathbf{k}-\mathbf{c}} \gamma(t_1) \right),
e^{\frac{1}{2} i (t-t_1) \Delta_{\pm}} \gamma (t_1)
 \right) dt_1. \\
\end{aligned}
\end{equation}
We deduce the following estimates using 
Proposition \ref{prop:KeyProp}:
\begin{equation}
\label{eq:bd50}
\begin{aligned}
& \left\Vert \partial_*^{\mathbf{k}} \gamma \right\Vert_{L^\infty_T H^{\alpha,\beta}}
\leq \left\Vert \partial_*^{\mathbf{k}} \gamma(0) \right\Vert_{H^{\alpha,\beta}}
+ \left\Vert \zeta_{\mathbf{k}} \right\Vert_{L^1_T H^{\alpha,\beta}}  + \\
& \qquad \qquad \qquad \qquad \qquad \qquad
+ C_{\mathbf{k}} \sup_{\substack{\mathbf{c}\leq \mathbf{k} \\ \mathbf{c} \neq 0 \\
\mathbf{c} \neq \mathbf{k}}}
\left\Vert B \left( \partial_*^{\mathbf{c}} \gamma,
\partial_*^{\mathbf{k}-\mathbf{c}} \gamma \right) \right\Vert_{L^1_T H^{\alpha,\beta}}
\end{aligned}
\end{equation}
\begin{equation}
\label{eq:bd51}
\begin{aligned}
& \left\Vert \zeta_{\mathbf{k}} \right\Vert_{L^1_T H^{\alpha,\beta}} \leq
C T^{\frac{1}{2}} \left\Vert \gamma \right\Vert_{L^\infty_T H^{\alpha,\beta}}
\left\Vert \partial_*^{\mathbf{k}} \gamma(0) \right\Vert_{H^{\alpha,\beta}} + \\
& \qquad \qquad \qquad
+  C T^{\frac{1}{2}} \left\Vert \partial_*^{\mathbf{k}} \gamma
\right\Vert_{L^\infty_T H^{\alpha,\beta}}
\left\Vert \zeta \right\Vert_{L^1_T H^{\alpha,\beta}} + \\
& \qquad \qquad \qquad
+ C T^{\frac{1}{2}} \left\Vert \gamma \right\Vert_{L^\infty_T H^{\alpha,\beta}}
\left\Vert \zeta_{\mathbf{k}} \right\Vert_{L^1_T H^{\alpha,\beta}} + \\
& \qquad \qquad \qquad + C_{\mathbf{k}} T^{\frac{1}{2}} \left\Vert \gamma 
\right\Vert_{L^\infty_T H^{\alpha,\beta}}
\sup_{\substack{\mathbf{c}\leq \mathbf{k} \\
\mathbf{c}\neq 0 \\ \mathbf{c}\neq \mathbf{k}}}
\left\Vert B\left( \partial_*^{\mathbf{c}} \gamma,
\partial_*^{\mathbf{k}-\mathbf{c}} \gamma \right) \right\Vert_{L^1_T H^{\alpha,\beta}}.
\end{aligned}
\end{equation}
Combining \eqref{eq:bd50}-\eqref{eq:bd51} with \eqref{eq:bd55}, and choosing $T$ sufficiently
small depending only on $\left\Vert \gamma \right\Vert_{L^\infty_T H^{\alpha,\beta}}$
and $\left\Vert \zeta \right\Vert_{L^1_T H^{\alpha,\beta}}$
(and noting that $\left\Vert \zeta \right\Vert_{L^1_T H^{\alpha,\beta}}$ scales at worst like
$T^{\frac{1}{2}}$ for $T$ small), we obtain:
\begin{equation}
\partial_*^{\mathbf{k}} \gamma \in L^\infty_T H^{\alpha,\beta}
\end{equation}
\begin{equation}
\zeta_{\mathbf{k}} \in L^1_T H^{\alpha,\beta}
\end{equation}
Letting $\mathbf{k}$ range over multi-indices of magnitude $|\mathbf{k}|$, we are able
to conclude that $\gamma \in L^\infty_T H^{\alpha+|\mathbf{k}|,\beta}$. 

Finally we must verify the assertion (\ref{eq:bd55}) to use for the next step of the
induction. This follows immediately from
Proposition \ref{prop:KeyProp} combined with the facts that
$\partial_*^{\mathbf{k}} \gamma \in L^\infty_T H^{\alpha,\beta}$ and
$\zeta_{\mathbf{k}} \in L^1_T H^{\alpha,\beta}$.
\end{proof}

\begin{proposition}
Let $\gamma (t)$ be a solution of Boltzmann's equation with
$\left\Vert \mathbf{b} \right\Vert_\infty < \infty$, and suppose
$\gamma \in L^\infty_T H^{\alpha,\beta}$ and $B(\gamma,\gamma) \in L^1_T H^{\alpha,\beta}$
for some $\alpha,\beta > \frac{d-1}{2}$, and further suppose that
$\gamma (0) \in H^{\alpha,\beta+r}$ for some $r>0$. Then
$\gamma \in L^\infty_T H^{\alpha,\beta+r}$ and $B(\gamma,\gamma) \in L^1_T H^{\alpha,\beta+r}$.
\end{proposition}
\begin{proof}
As usual, it suffices to prove the result for a small time depending on
$\left\Vert \gamma \right\Vert_{L^\infty_T H^{\alpha,\beta}}$ and
$\left\Vert B(\gamma,\gamma) \right\Vert_{L^1_T H^{\alpha,\beta}}$. Furthermore, we will
only prove the result for small values of $r$ (with smallness only depending on
$d,\alpha,\beta$), as allowed by
Proposition \ref{prop:KeyProp}; the general result then follows immediately.

We know $\gamma$ solves Boltzmann's equation,
\begin{equation}
\left( i \partial_t + \frac{1}{2} (\Delta_x - \Delta_{x^\prime}) \right) \gamma
= B (\gamma,\gamma)
\end{equation}
This equation is equivalent to the following system of equations, where we write
$\zeta (t) = B(\gamma(t),\gamma(t))$:
\begin{equation}
\gamma (t) = e^{\frac{1}{2} i t \Delta_{\pm}} \gamma (0) - i \int_0^t
e^{\frac{1}{2} i (t-t_1) \Delta_{\pm}} \zeta(t_1) dt_1
\end{equation}
\begin{equation}
\begin{aligned}
& \zeta (t) = 
B \left( e^{\frac{1}{2} i t \Delta_{\pm}} \gamma (0),
e^{\frac{1}{2} i t \Delta_{\pm}} \gamma (0) \right) \\
& \qquad \qquad \qquad - i \int_0^t 
B\left( e^{\frac{1}{2} i (t-t_1) \Delta_{\pm}} \gamma(t_1),
e^{\frac{1}{2} i (t-t_1) \Delta_{\pm}} \zeta(t_1) \right) dt_1 \\
& \qquad \qquad \qquad - i \int_0^t 
B\left( e^{\frac{1}{2} i (t-t_1) \Delta_{\pm}} \zeta(t_1),
e^{\frac{1}{2} i (t-t_1) \Delta_{\pm}} \gamma(t_1) \right) dt_1 \\
\end{aligned}
\end{equation}
Applying Proposition \ref{prop:KeyProp} with a small $\delta > 0$, we obtain:
\begin{equation}
\begin{aligned}
& \left\Vert \gamma \right\Vert_{L^\infty_T H^{\alpha,\beta+\delta}} \leq
\left\Vert \gamma(0) \right\Vert_{H^{\alpha,\beta+\delta}} +
\left\Vert \zeta \right\Vert_{L^1_T H^{\alpha,\beta+\delta}}
\end{aligned}
\end{equation}
\begin{equation}
\begin{aligned}
& \left\Vert \zeta \right\Vert_{L^1_T H^{\alpha,\beta+\delta}} \leq
C T^{\frac{1}{2}} \left\Vert \gamma(0) \right\Vert_{H^{\alpha,\beta+\delta}}
\left\Vert \gamma (0) \right\Vert_{H^{\alpha,\beta}} + \\
&\qquad + C T^{\frac{1}{2}} \left\Vert \gamma \right\Vert_{L^\infty_T H^{\alpha,\beta+\delta}}
\left\Vert \zeta \right\Vert_{L^1_T H^{\alpha,\beta}} +
 C T^{\frac{1}{2}} \left\Vert \zeta \right\Vert_{L^1_T H^{\alpha,\beta+\delta}}
\left\Vert \gamma \right\Vert_{L^\infty_T H^{\alpha,\beta}}
\end{aligned}
\end{equation}
Since $\left\Vert \zeta \right\Vert_{L^1_T H^{\alpha,\beta}}$ scales at worst like
$\mathcal{O} (T^{\frac{1}{2}})$, we easily deduce that
$\gamma \in L^\infty_T H^{\alpha,\beta+\delta}$ and $\zeta \in L^1_T H^{\alpha,\beta+\delta}$,
as soon as $T$ is chosen sufficiently small depending only on
the norm of $\gamma$ in $H^{\alpha,\beta}$.
\end{proof}

\section{Regularity in time}
\label{sec:treg}

In this section we discuss regularity in the \emph{time} variable.
This should seem to be a very simple matter due to the obvious formula
\begin{equation}
\partial_t B(\gamma,\gamma) = B ( \partial_t \gamma, \gamma) +
B(\gamma, \partial_t \gamma)
\end{equation}
and the fact that $\partial_t$ commutes with 
$\left( i \partial_t + \frac{1}{2} (\Delta_x - \Delta_{x^\prime}) \right)$. The difficulty
which arises can already be seen when we try to control
$\partial_t \gamma$ in $L^\infty_T H^{\alpha,\beta}$ (in fact we will need to control higher
derivatives $\partial^k_t \gamma$ to eventually prove that $\gamma$ is locally 
$C^r$ in $(t,x,x^\prime)$). Indeed we have the following closed equation for
$\partial_t \gamma$:
\begin{equation}
\left( i \partial_t + \frac{1}{2} (\Delta_x - \Delta_{x^\prime}) \right)
\left( \partial_t \gamma \right) =
B ( \partial_t \gamma, \gamma) + B(\gamma,\partial_t \gamma)
\end{equation}
We can re-cast this system as a closed pair of integral equations for the
functions $\left\{ \partial_t \gamma, \partial_t \zeta \right\}$ where
$\zeta = B(\gamma,\gamma) $, exactly as we have done previously to propagate regularity and
moments in all other variables.

However, to solve the system, we will at the very least need to know that
$(\partial_t \gamma) (0) \in H^{\alpha,\beta}$, which means
\begin{equation}
\frac{i}{2} (\Delta_x - \Delta_{x^\prime}) \gamma (0) - i B(\gamma(0),\gamma (0))
\in H^{\alpha,\beta} 
\end{equation}
For the first term it is enough to suppose that
$\gamma(0) \in H^{\alpha+1,\beta+1}$ (say), but the second term is tricky because
Proposition \ref{prop:KeyProp} does not provide bounds for $B(\gamma,\gamma)$ in
the spaces $H^{\alpha,\beta}$. 
 Only quantities involving the free propagator, such as
$B\left(e^{\frac{1}{2} i t \Delta_{\pm}} \gamma(0),
 e^{\frac{1}{2} i t \Delta_{\pm}} \gamma(0) \right)$, are
controlled via Proposition \ref{prop:KeyProp}. 
To resolve this problem, in Appendix \ref{sec:AppA} we prove bilinear estimates for
$B(\gamma,\gamma)$ \emph{without} a free propagator; the price we must pay is a small loss
 in the $\beta$ index for the gain term (and we must also assume $\alpha,\beta > \frac{d}{2}$).
Due to the loss coming from Proposition \ref{prop:AppA-lossy}, it is not possible to use
that bound in place of Proposition \ref{prop:KeyProp} elsewhere in this paper.

\begin{proposition}
Let $\gamma (t)$ be a solution of Boltzmann's equation with $\left\Vert \mathbf{b}
\right\Vert_\infty < \infty$, and suppose $\gamma \in L^\infty_T H^{\alpha,\beta}$
and $B(\gamma,\gamma) \in L^1_T H^{\alpha,\beta}$ for some $\alpha,\beta > \frac{d-1}{2}$.
Further suppose that $K>0$ is an integer with $K < \min(\alpha,\beta) - \frac{d}{2}$.
 Then for all $1 \leq k \leq K$ there
holds $\partial_t^k \gamma \in L^\infty_T H^{\alpha-k,\beta-k}$.
\end{proposition}
\begin{proof}
We will prove the desired result on a small time interval $T>0$ depending only on the
underlying solution $\gamma(t)$ of Boltzmann's equation; the general result then follows
immediately.

For any integer $k \geq 1$ we have
\begin{equation}
\label{eq:dkteq}
\begin{aligned}
& \left( i \partial_t + \frac{1}{2} (\Delta_x - \Delta_{x^\prime}) \right)
\left( \partial_t^k \gamma \right)
= B (\partial_t^k \gamma,\gamma) + B(\gamma,\partial_t^k \gamma) + \\
&\qquad \qquad \qquad \qquad
\qquad \qquad
\qquad \qquad
 + \sum_{0<j<k} \binom{k}{j} B(\partial_t^{k-j} \gamma, \partial_t^j \gamma)
\end{aligned}
\end{equation}
Let us define
\begin{equation}
\zeta_k (t) = B ( \partial_t^k \gamma(t),\gamma(t)) +
B( \gamma(t), \partial_t^k \gamma(t))
\end{equation}
Then (\ref{eq:dkteq}) is equivalent to the following system of equations:
\begin{equation}
\begin{aligned}
& \partial_t^k \gamma (t) =
e^{\frac{1}{2} i t \Delta_{\pm}} \partial^k_t \gamma (0)
 - i \int_0^t e^{\frac{1}{2} i (t-t_1) \Delta_{\pm}} \zeta_k (t_1) dt_1 \\
& - i \sum_{0 < j < k} \binom{k}{j}
\int_0^t e^{\frac{1}{2} i (t-t_1) \Delta_{\pm}}
B \left( \partial_t^{k-j} \gamma (t_1),
\partial_t^j \gamma (t_1) \right) dt_1
\end{aligned}
\end{equation}
\begin{equation}
\begin{aligned}
& \zeta_k (t) =
B \left( e^{\frac{1}{2} i t \Delta_{\pm}} \partial_t^k \gamma(0),
e^{\frac{1}{2} i t \Delta_{\pm}} \gamma(0) \right) +
B \left( e^{\frac{1}{2} i t \Delta_{\pm}} \gamma(0),
e^{\frac{1}{2} i t \Delta_{\pm}} \partial_t^k \gamma(0) \right) \\
& - i \int_0^t B \left( e^{\frac{1}{2} i (t-t_1) \Delta_{\pm}}
\zeta_k (t_1), e^{\frac{1}{2} i (t-t_1) \Delta_{\pm}} \gamma (t_1) \right) dt_1 \\
& - i \int_0^t B \left( e^{\frac{1}{2} i (t-t_1) \Delta_{\pm}}
\partial_t^k \gamma (t_1), e^{\frac{1}{2} i (t-t_1) \Delta_{\pm}} \zeta (t_1) \right) dt_1 \\
& - i \int_0^t B \left( e^{\frac{1}{2} i (t-t_1) \Delta_{\pm}}
\zeta (t_1), e^{\frac{1}{2} i (t-t_1) \Delta_{\pm}} \partial_t^k \gamma (t_1) \right) dt_1 \\
& - i \int_0^t B \left( e^{\frac{1}{2} i (t-t_1) \Delta_{\pm}}
\gamma (t_1), e^{\frac{1}{2} i (t-t_1) \Delta_{\pm}} \zeta_k (t_1) \right) dt_1 \\
& - i \sum_{0<j<k} \binom{k}{j} \times \\
& \qquad \quad \times \int_0^t B \left( e^{\frac{1}{2} i (t-t_1) \Delta_{\pm}}
B\left( \partial_t^{k-j} \gamma (t_1),\partial_t^j \gamma (t_1) \right),
 e^{\frac{1}{2} i (t-t_1) \Delta_{\pm}} \gamma (t_1) \right) dt_1 \\
& - i \sum_{0<j<k} \binom{k}{j} \times \\
& \qquad \quad \times \int_0^t B \left( e^{\frac{1}{2} i (t-t_1) \Delta_{\pm}}
\gamma (t_1),
 e^{\frac{1}{2} i (t-t_1) \Delta_{\pm}} 
B\left( \partial_t^{k-j} \gamma (t_1),\partial_t^j \gamma (t_1) \right)
 \right) dt_1 \\
\end{aligned}
\end{equation}
Hence by applying Proposition (\ref{prop:KeyProp}), we obtain the following estimates:
\begin{equation}
\begin{aligned}
& \left\Vert \partial_t^k \gamma(t) \right\Vert_{L^\infty_T H^{\alpha-k,\beta-k}} \leq
\left\Vert \partial_t^k \gamma (0) \right\Vert_{H^{\alpha-k,\beta-k}}
+ \left\Vert \zeta_k \right\Vert_{L^1_T H^{\alpha-k,\beta-k}} \\
& \qquad \qquad \qquad + \sum_{0 < j < k} \binom{k}{j}
\left\Vert B \left(
\partial_t^{k-j} \gamma (t_1), \partial_t^j \gamma (t_1) \right)
\right\Vert_{L^1_T H^{\alpha-k,\beta-k}}
\end{aligned}
\end{equation}
\begin{equation}
\begin{aligned}
& \left\Vert \zeta_k \right\Vert_{L^1_T H^{\alpha-k,\beta-k}} \leq
C T^{\frac{1}{2}} \left\Vert \partial_t^k \gamma(0) 
\right\Vert_{H^{\alpha-k,\beta-k}} \left\Vert \gamma(0)
\right\Vert_{H^{\alpha,\beta}} + \\
& + C T^{\frac{1}{2}} \left\Vert \zeta_k \right\Vert_{L^1_T H^{\alpha-k,\beta-k}}
\left\Vert \gamma \right\Vert_{L^\infty_T H^{\alpha,\beta}}
+ C T^{\frac{1}{2}} \left\Vert \partial_t^k \gamma \right\Vert_{L^\infty_T H^{\alpha-k,\beta-k}}
\left\Vert \zeta \right\Vert_{L^1_T H^{\alpha,\beta}} + \\
& + C_k T^{\frac{1}{2}} \left\Vert \gamma \right\Vert_{L^\infty_T H^{\alpha,\beta}}
\sup_{0 < j < k}
\left\Vert B \left( \partial_t^{k-j} \gamma(t),\partial_t^j \gamma(t) \right)
\right\Vert_{L^1_T H^{\alpha-k,\beta-k}}
\end{aligned}
\end{equation}
Now if we assume that $\partial_t^k \gamma(0) \in H^{\alpha-k,\beta-k}$,
$\partial_t^j \gamma \in L^\infty_T H^{\alpha-j,\beta-j}$ for $0\leq j < k$,
and $\zeta_j \in L^1_T H^{\alpha-j,\beta-j}$ for $0\leq j < k$, then we can
show that $\partial_t^k \gamma \in L^\infty_T H^{\alpha-k,\beta-k}$
and $\zeta_k \in L^1_T H^{\alpha-k,\beta-k}$.

It only remains to verify that $\partial_t^k \gamma (0) \in H^{\alpha-k,\beta-k}$. To 
see this, observe that
\begin{equation}
i \partial_t^k \gamma + \frac{1}{2} (\Delta_x - \Delta_{x^\prime})
\partial_t^{k-1} \gamma =
\sum_{0 \leq j \leq k-1} \binom{k-1}{j} 
B ( \partial_t^{k-1-j} \gamma, \partial_t^j \gamma)
\end{equation}
Now since $\partial_t^{k-1} \gamma \in L^\infty_T H^{\alpha-k+1,\beta-k+1}$ (in fact it is
continuous in time in this functional space), we have $(\Delta_x - \Delta_{x^\prime})
\left(\partial_t^{k-1} \gamma\right)(0) \in H^{\alpha-k,\beta-k}$. Additionally, 
since $\left( \partial_t^j  \gamma (0)\right) \in H^{\alpha-j,\beta-j}$ for $0\leq j < k$,
by Proposition \ref{prop:AppA-lossy}, we find that
$B\left( (\partial_t^{k-1-j} \gamma) (0), (\partial_t^j \gamma(0)) \right) \in
H^{\alpha-k,\beta-k}$ when $0 \leq j \leq k-1$.
\end{proof}

\section{Continuity of the Solution Map}
\label{sec:cont}

It is sometimes useful to be able to approximate a given initial data by some other,
better behaved, initial data for the purpose of formal computations. In order to
pass the results of computations to the limit and reach a non-void conclusion about the original
data, it is necessary to know that the solution map is sufficiently smooth with respect
to perturbations. Such an argument is apparently necessary to prove the non-negativity
of solutions to Boltzmann's equation in the spaces $H^{\alpha,\beta}$ for arbitrary
$\alpha,\beta > \frac{d-1}{2}$, because we have no convenient
characterization of non-negativity of
$f$ just from looking at the (inverse) Wigner transform $\gamma$.

\begin{proposition}
Let $\gamma^j (t)$ be a solution of Boltzmann's equation with
$\left\Vert \mathbf{b}\right\Vert_\infty < \infty$, for $j=1,2$, with
$\gamma^j \in L^\infty_T H^{\alpha,\beta}$ and
$B(\gamma^j,\gamma^j) \in L^1_T H^{\alpha,\beta}$ for $j=1,2$ and
some $\alpha,\beta > \frac{d-1}{2}$. Furthermore, assume that
$\left\Vert \gamma^j \right\Vert_{L^\infty_T H^{\alpha,\beta}} \leq M$ for
$j=1,2$. Then we have
\begin{equation}
\left\Vert \gamma^1 - \gamma^2 \right\Vert_{L^\infty_T H^{\alpha,\beta}} \leq
C_{M,T} \left\Vert \gamma^1 (0) - \gamma^2 (0) \right\Vert_{H^{\alpha,\beta}}
\end{equation}
where the constant may depend on $\alpha,\beta$.
\end{proposition}
\begin{proof}
We will prove the result for $T$ small enough depending only on the upper bound $M$; the
full result then follows by iterating in time. 

We recall that each $\gamma^i$ solves Boltzmann's equation:
\begin{equation}
\left( i \partial_t + \frac{1}{2} \left( \Delta_x - \Delta_{x^\prime} \right) \right)
\gamma^j = B \left( \gamma^j,\gamma^j \right)
\end{equation}
This is equivalent to the following system of equations, where
$\zeta^j (t) = B(\gamma^j (t),\gamma^j (t))$:
\begin{equation}
\begin{aligned}
& \gamma^j (t) = e^{\frac{1}{2} i t \Delta_{\pm}} \gamma^j (0) - i \int_0^t
e^{\frac{1}{2} i(t-t_1)\Delta_{\pm}} \zeta^j (t_1) dt_1
\end{aligned}
\end{equation}
\begin{equation}
\begin{aligned}
& \zeta^j (t) = B \left( e^{\frac{1}{2} i t \Delta_{\pm}} \gamma^j (0), 
e^{\frac{1}{2} i t \Delta_{\pm}}\gamma^j (0) \right) \\
& \qquad \qquad - i \int_0^t B \left( e^{\frac{1}{2} i (t-t_1) \Delta_{\pm}} \gamma^j (t_1),
e^{\frac{1}{2} i (t-t_1) \Delta_{\pm}} \zeta^j (t_1) \right) dt_1 \\
& \qquad \qquad - i \int_0^t B \left( e^{\frac{1}{2} i (t-t_1) \Delta_{\pm}} \zeta^j (t_1),
e^{\frac{1}{2} i (t-t_1) \Delta_{\pm}} \gamma^j (t_1) \right) dt_1
\end{aligned}
\end{equation}
Let us define
\begin{equation}
\gamma^r (t) = \gamma^1 (t) - \gamma^2 (t)
\end{equation}
\begin{equation}
\zeta^r (t) = \zeta^1 (t) - \zeta^2 (t)
\end{equation}
Then we can write the following system of equations for $\gamma^r , \zeta^r$:
\begin{equation}
\begin{aligned}
\gamma^r (t) = e^{\frac{1}{2} i t \Delta_{\pm}} \gamma^r (0) - i \int_0^t
e^{\frac{1}{2} i(t-t_1)\Delta_{\pm}} \zeta^r (t_1) dt_1
\end{aligned}
\end{equation}
\begin{equation}
\begin{aligned}
& \zeta^r (t) 
= B \left( e^{\frac{1}{2} i t \Delta_{\pm}} \gamma^1 (0), 
e^{\frac{1}{2} i t \Delta_{\pm}}\gamma^r (0) \right) +
 B \left( e^{\frac{1}{2} i t \Delta_{\pm}} \gamma^r (0), 
e^{\frac{1}{2} i t \Delta_{\pm}}\gamma^2 (0) \right)\\
& \qquad \qquad - i \int_0^t B \left( e^{\frac{1}{2} i (t-t_1) \Delta_{\pm}} \gamma^1 (t_1),
e^{\frac{1}{2} i (t-t_1) \Delta_{\pm}} \zeta^r (t_1) \right) dt_1 \\
& \qquad \qquad - i \int_0^t B \left( e^{\frac{1}{2} i (t-t_1) \Delta_{\pm}} \gamma^r (t_1),
e^{\frac{1}{2} i (t-t_1) \Delta_{\pm}} \zeta^2 (t_1) \right) dt_1 \\
& \qquad \qquad - i \int_0^t B \left( e^{\frac{1}{2} i (t-t_1) \Delta_{\pm}} \zeta^1 (t_1),
e^{\frac{1}{2} i (t-t_1) \Delta_{\pm}} \gamma^r (t_1) \right) dt_1 \\
& \qquad \qquad - i \int_0^t B \left( e^{\frac{1}{2} i (t-t_1) \Delta_{\pm}} \zeta^r (t_1),
e^{\frac{1}{2} i (t-t_1) \Delta_{\pm}} \gamma^2 (t_1) \right) dt_1
\end{aligned}
\end{equation}
Using Proposition \ref{prop:KeyProp}, we deduce the following estimates:
\begin{equation}
\left\Vert \gamma^r \right\Vert_{L^\infty_T H^{\alpha,\beta}} \leq
\left\Vert \gamma^r (0) \right\Vert_{H^{\alpha,\beta}} +
\left\Vert \zeta^r \right\Vert_{L^1_T H^{\alpha,\beta}}
\end{equation}
\begin{equation}
\begin{aligned}
& \left\Vert \zeta^r \right\Vert_{L^1_T H^{\alpha,\beta}} \leq
C T^{\frac{1}{2}} \left(
\left\Vert \gamma^1 \right\Vert_{L^\infty_T H^{\alpha,\beta}} +
\left\Vert \gamma^2 \right\Vert_{L^\infty_T H^{\alpha,\beta}} \right)
 \left\Vert \gamma^r (0) \right\Vert_{H^{\alpha,\beta}} + \\
& \qquad + CT^{\frac{1}{2}} \left\Vert \gamma^1 \right\Vert_{L^\infty_T H^{\alpha,\beta}}
\left\Vert \zeta^r \right\Vert_{L^1_T H^{\alpha,\beta}}
+ C T^{\frac{1}{2}} \left\Vert \gamma^r \right\Vert_{L^\infty_T H^{\alpha,\beta}}
\left\Vert \zeta^2 \right\Vert_{L^1_T H^{\alpha,\beta}} + \\
& \qquad + C T^{\frac{1}{2}} \left\Vert \zeta^1 \right\Vert_{L^1_T H^{\alpha,\beta}}
\left\Vert \gamma^r \right\Vert_{L^\infty_T H^{\alpha,\beta}} +
C T^{\frac{1}{2}} \left\Vert \zeta^r \right\Vert_{L^1_T H^{\alpha,\beta}}
\left\Vert \gamma^2 \right\Vert_{L^\infty_T H^{\alpha,\beta}}
\end{aligned}
\end{equation}
Hence if we define
\begin{equation}
A_T = T^{\frac{1}{2}} \left\Vert \gamma^r \right\Vert_{L^\infty_T H^{\alpha,\beta}}
+ \left\Vert \zeta^r \right\Vert_{L^1_T H^{\alpha,\beta}}
\end{equation}
then we easily deduce
\begin{equation}
A_T \leq C_M T^{\frac{1}{2}} A_T + C_M T^{\frac{1}{2}}
\left\Vert \gamma^r (0) \right\Vert_{H^{\alpha,\beta}}
\end{equation}
Choosing $T$ sufficiently small (depending only on $M$), the conclusion follows.
\end{proof}

\appendix

\section{Bilinear Estimates with Loss}
\label{sec:AppA}

The main difficulty in proving time regularity estimates in Section \ref{sec:treg}
is the fact that Proposition \ref{prop:KeyProp} only controls the collision operator
in $L^1_T$, whereas we must make sense of $B(\gamma,\gamma)$ at a \emph{fixed}
time (namely $t=0$) just to write down $\left(\partial_t \gamma\right) (0)$. In order to resolve this
issue, in this Appendix we prove instantaneous bounds on the operator
$B(\gamma,\gamma)$ in the Sobolev spaces $H^{\alpha,\beta}$ when $\alpha, \beta >
\frac{d}{2}$. The proof involves a small loss in the $\beta$ index;
 for this reason, these estimates (as stated in this Appendix) \emph{cannot} 
replace Proposition \ref{prop:KeyProp} elsewhere in this paper, regardless of the
size of $\alpha,\beta$.
  The idea of the proof is drawn from Theorem 4.3 of
\cite{PC2011}.

\begin{proposition}
\label{prop:AppA-lossy}
Let $\alpha,\beta \in \left( \frac{d}{2},\infty\right)$. Then for any $\delta > 0$ we
have for all $\gamma_1,\gamma_2 \in H^{\alpha,\beta+\delta}$ the following estimates:
\begin{equation}
\left\Vert B^{\pm} \left( \gamma_1,\gamma_2 \right) \right\Vert_{H^{\alpha,\beta}} 
\leq C_\delta \left\Vert \mathbf{b} \right\Vert_\infty
\left\Vert \gamma_1 \right\Vert_{H^{\alpha,\beta+\delta}}
\left\Vert \gamma_2 \right\Vert_{H^{\alpha,\beta+\delta}}
\end{equation}
\end{proposition}
\begin{proof}
We treat the loss term and gain term separately.

\textbf{Loss term.} For any function $F(x,x^\prime)$ we denote the Fourier transform,
\begin{equation}
\hat{F} (\xi,\xi^\prime) = \int dx dx^\prime e^{- i \xi \cdot x}
e^{-i \xi^\prime \cdot x^\prime} F(x,x^\prime)
\end{equation}
A routine computation shows that, if $\left\Vert \mathbf{b} \right\Vert_\infty < \infty$,
then
\begin{equation}
\left| {\left( B^- (\gamma_1,\gamma_2) \right)}^\wedge (\xi,\xi^\prime) \right| \leq
C \left\Vert \mathbf{b} \right\Vert_\infty \int d\eta d\eta^\prime
\left| \hat{\gamma}_1 \left(\xi-\frac{\eta+\eta^\prime}{2}\right)\right|
\left| \hat{\gamma}_2 (\eta,\eta^\prime) \right|
\end{equation}
Therefore we have
\begin{equation}
\begin{aligned}
& \left\Vert B^- (\gamma_1,\gamma_2) \right\Vert_{H^{\alpha,\beta}}^2 =
\int \left| \left( B^- (\gamma_1,\gamma_2) \right)^\wedge (\xi,\xi^\prime)\right|^2
\left< \xi+\xi^\prime \right>^{2\alpha} 
\left< \xi-\xi^\prime \right>^{2\beta} d\xi d\xi^\prime \\
& \leq C^2 \left\Vert \mathbf{b} \right\Vert_\infty^2
\int d\xi d\xi^\prime d\eta_1 d\eta_1^\prime d\eta_2 d\eta_2^\prime
\left< \xi+\xi^\prime \right>^{2\alpha} 
\left< \xi-\xi^\prime \right>^{2\beta} \times \\
&\qquad \qquad \qquad \times \left| \hat{\gamma}_1 \left( \xi-\frac{\eta_1+\eta_1^\prime}{2},
\xi^\prime - \frac{\eta_1+\eta_1^\prime}{2}\right) \right|
\left| \hat{\gamma}_2 (\eta_1,\eta_1^\prime)\right| \times \\
& \qquad \qquad \qquad \times \left| \hat{\gamma}_1 \left( \xi-\frac{\eta_2+\eta_2^\prime}{2},
\xi^\prime - \frac{\eta_2+\eta_2^\prime}{2}\right) \right|
\left| \hat{\gamma}_2 (\eta_2,\eta_2^\prime)\right| 
\end{aligned}
\end{equation}
Now we multiply and divide under the integral sign by the following factor,
\begin{equation}
\prod_{i=1,2} \left\{ \left< \xi+\xi^\prime-\eta_i-\eta_i^\prime \right>^\alpha
\left< \xi-\xi^\prime \right>^\beta
\left< \eta_i+\eta_i^\prime \right>^\alpha
\left< \eta_i-\eta_i^\prime \right>^\beta\right\}
\end{equation}
and apply the Cauchy-Schwarz inequality \emph{pointwise} under the integral. This gives us
\begin{equation}
\begin{aligned}
& \left\Vert B^- (\gamma_1,\gamma_2) \right\Vert_{H^{\alpha,\beta}}^2 \leq
C^2 \left\Vert \mathbf{b} \right\Vert_\infty^2 \int d\xi d\xi^\prime
d\eta_1 d\eta_1^\prime d\eta_2 d\eta_2^\prime \times \\
& \qquad \qquad \qquad \qquad \times \frac{\left< \xi+\xi^\prime\right>^{2\alpha}}
{\left< \xi+\xi^\prime-\eta_2-\eta_2^\prime \right>^{2\alpha}
\left< \eta_2+\eta_2^\prime \right>^{2\alpha}
\left< \eta_2-\eta_2^\prime\right>^{2\beta}} \times \\
& \qquad \times \left| \left< \xi+\xi^\prime-\eta_1-\eta_1^\prime \right>^\alpha
\left< \xi-\xi^\prime \right>^\beta \hat{\gamma}_1 \left(
\xi-\frac{\eta_1+\eta_1^\prime}{2},\xi^\prime-\frac{\eta_1+\eta_1^\prime}{2}\right)\right|^2 \times \\
& \qquad \times \left| \left< \eta_1+\eta_1^\prime \right>^\alpha
\left< \eta_1-\eta_1^\prime \right>^\beta
\hat{\gamma}_2 (\eta_1,\eta_1^\prime) \right|^2
\end{aligned}
\end{equation}
Observe now that if
\begin{equation}
\label{eq:I-loss-bd}
I \equiv \sup_{\xi,\xi^\prime} \int d\eta d\eta^\prime
\frac{\left< \xi+\xi^\prime\right>^{2\alpha}}
{\left< \xi+\xi^\prime-\eta_2-\eta_2^\prime \right>^{2\alpha}
\left< \eta_2+\eta_2^\prime \right>^{2\alpha}
\left< \eta_2-\eta_2^\prime\right>^{2\beta}} < \infty
\end{equation}
then we conclude
\begin{equation}
\left\Vert B^- (\gamma_1,\gamma_2) \right\Vert_{H^{\alpha,\beta}}
\leq C \left\Vert \mathbf{b}\right\Vert_\infty
\left\Vert \gamma_1 \right\Vert_{H^{\alpha,\beta}} 
\left\Vert \gamma_2 \right\Vert_{H^{\alpha,\beta}}
\end{equation}

The bound (\ref{eq:I-loss-bd}) is equivalent to the following estimate:
\begin{equation}
\label{eq:I-loss-bd-2}
\sup_{W \in \mathbb{R}^d} \int_{\mathbb{R}^d} dw 
\frac{\left<W\right>^{2\alpha}}
{\left<W-w\right>^{2\alpha} \left<w \right>^{2\alpha}}
\int_{\mathbb{R}^d} dz \frac{1}{\left< z \right>^{2\beta}} < \infty
\end{equation}
It is easy to verify that the bound (\ref{eq:I-loss-bd-2}) holds whenever
$\alpha,\beta > \frac{d}{2}$.

\textbf{Gain term.} By a routine calculation, if $\left\Vert \mathbf{b}
\right\Vert_\infty < \infty$, we have:
\begin{equation}
\begin{aligned}
& \left| \left( B^+ (\gamma_1,\gamma_2) \right)^\wedge (\xi,\xi^\prime) \right| \leq \\
& \leq C \left\Vert \mathbf{b} \right\Vert_\infty
\int_{\mathbb{S}^{d-1}} d\omega
\int d\eta_1 d\eta_1^\prime d\eta_2 d\eta_2^\prime
\hat{\gamma}_1 (\eta_1,\eta_1^\prime) \hat{\gamma}_2 (\eta_2,\eta_2^\prime) \times \\
& \times \delta \left( -\xi + \eta_1 + \frac{\eta_2 + \eta_2^\prime}{2}
- \frac{1}{2} P_\omega (\eta_1-\eta_1^\prime) 
+ \frac{1}{2} P_\omega (\eta_2-\eta_2^\prime) \right) \times \\
& \times \delta \left( -\xi^\prime + \eta_1^\prime + \frac{\eta_2+\eta_2^\prime}{2}
+\frac{1}{2} P_\omega (\eta_1-\eta_1^\prime) - \frac{1}{2} P_\omega (\eta_2-\eta_2^\prime)
\right)
\end{aligned}
\end{equation}
Performing changes of variables as in \cite{CDP2017}, this gives:
\begin{equation}
\begin{aligned}
& \left| \left( B^+ (\gamma_1,\gamma_2) \right)^\wedge (\xi,\xi^\prime) \right| \leq \\
& \qquad \leq C \left\Vert \mathbf{b} \right\Vert_\infty
\int_{\mathbb{S}^{d-1}} d\omega
\int ds_1 ds_2 \times \\
& \qquad \qquad \times \hat{\gamma}_1 \left(
s_1 + 2 s_2^\Vert + \frac{3\xi-\xi^\prime}{4},
s_1-2 s_2^\Vert + \frac{3\xi^\prime-\xi}{4}\right)\times \\
& \qquad \qquad \times \hat{\gamma}_2 \left( -s_1-2s_2^\bot + \frac{3\xi-\xi^\prime}{4},
-s_1+2s_2^\bot + \frac{3\xi^\prime-\xi}{4} \right)
\end{aligned}
\end{equation}
where $s_2^\Vert = P_\omega (s_2)$ and $s_2^\bot = (\mathbb{I}-P_\omega)(s_2)$.

Reasoning as for the loss term, if we can show that the integral
\begin{equation}
\begin{aligned}
& \int_{\mathbb{S}^{d-1}} d\omega \int_{\mathbb{R}^d \times \mathbb{R}^d}
 ds_1 ds_2 \times \\
& \times
\frac{\left<\xi+\xi^\prime\right>^{2\alpha} \left<\xi-\xi^\prime\right>^{2\beta}}
{\left< 2s_1+\frac{\xi+\xi^\prime}{2}\right>^{2\alpha}
\left< 4s_2^\Vert + \xi - \xi^\prime \right>^{2(\beta+\delta)}
\left< -2s_1 + \frac{\xi+\xi^\prime}{2} \right>^{2\alpha}
\left< -4s_2^\bot + \xi-\xi^\prime \right>^{2(\beta+\delta)} }
\end{aligned}
\end{equation}
is bounded uniformly with respect to $\xi,\xi^\prime \in \mathbb{R}^d$, then we will
have the estimate
\begin{equation}
\left\Vert B^+ (\gamma_1,\gamma_2) \right\Vert_{H^{\alpha,\beta}} \leq
C \left\Vert \mathbf{b} \right\Vert_\infty
\left\Vert \gamma_1 \right\Vert_{H^{\alpha,\beta+\delta}}
\left\Vert \gamma_2 \right\Vert_{H^{\alpha,\beta+\delta}}
\end{equation}
In fact it suffices to show the following two bounds:
\begin{equation}
\label{eq:I20}
\sup_{W \in \mathbb{R}^d} \int_{\mathbb{R}^d}
ds \frac{\left< W \right>^{2\alpha}}
{\left< s \right>^{2\alpha} \left< W-s \right>^{2\alpha}} < \infty
\end{equation}
\begin{equation}
\label{eq:I21}
\sup_{W\in \mathbb{R}^d}
\int_{\mathbb{S}^{d-1}} d\omega \int_{\mathbb{R}^d} ds
\frac{\left<W\right>^{2\beta}}
{\left< s^\Vert + W^\bot \right>^{2(\beta+\delta)} 
\left< s^\bot + W^\Vert \right>^{2(\beta+\delta)}}
< \infty
\end{equation}
It is easy to verify that (\ref{eq:I20}) holds whenever $\alpha > \frac{d}{2}$, so
we only have to prove (\ref{eq:I21}).

We easily bound the integral (\ref{eq:I21}) with respect to $s\in\mathbb{R}^d$ when
$\beta > \frac{d-1}{2}$ by decomposing $ds = ds^\Vert ds^\bot$. Then all we must
show is that
\begin{equation}
\sup_{W \in \mathbb{R}^d} \int_{\mathbb{S}^{d-1}} d\omega
\left< W \right>^{2\beta}
\left< W^\Vert \right>^{d-1-2(\beta+\delta)} \left< W^\bot \right>^{1-2(\beta+\delta)}
<\infty
\end{equation}
Now we may decompose
\begin{equation}
\left< W \right>^{2\beta} \lesssim
\left< W^\Vert \right>^{2\beta} + \left< W^\bot \right>^{2\beta}
\end{equation}
Therefore it suffices to bound the following two integrals:
\begin{equation}
I = \int_{\mathbb{S}^{d-1}} d\omega
\left< W^\Vert \right>^{d-1-2\delta} \left< W^\bot \right>^{1-2(\beta+\delta)}
\end{equation}
\begin{equation}
I^\prime = \int_{\mathbb{S}^{d-1}} d\omega
\left< W^\Vert \right>^{d-1-2(\beta+\delta)} \left< W^\bot \right>^{1-2\delta}
\end{equation}
We can bound both integrals using a dyadic decomposition, as in \cite{CDP2017},
whenever $\beta > \frac{d}{2}$, as follows:
\begin{equation}
\begin{aligned}
I & \lesssim \sum_{k=1}^\infty \int_{\omega : 2^{-k-1} |W^\Vert| \leq
|W^\bot| < 2^{-k} |W^\Vert|} d\omega
\left< W^\Vert \right>^{d-1-2\delta} \left< W^\bot\right>^{1-2(\beta+\delta)} \\
& \lesssim \sum_{k=1}^\infty
2^{-k-1} \times (2^{-k})^{d-2} \times
(2^{k+1})^{d-1-2\delta} < \infty 
\end{aligned}
\end{equation}
\begin{equation}
\begin{aligned}
I^\prime & \lesssim \sum_{k=1}^\infty \int_{\omega : 2^{-k-1} |W^\bot| \leq
|W^\Vert| < 2^{-k} |W^\bot|} d\omega
\left< W^\Vert \right>^{d-1-2(\beta+\delta)} \left< W^\bot \right>^{1-2\delta} \\
& \lesssim \sum_{k=1}^\infty 2^{-k-1} \times
(2^{k+1})^{1-2\delta} < \infty
\end{aligned}
\end{equation}
Hence we may conclude.
\end{proof}

\section{Proof of Proposition \ref{prop:KeyProp} when $\delta = 0$}
\label{sec:AppB}

In this appendix we will provide a proof of Proposition \ref{prop:KeyProp};
it is based on bilinear Strichartz estimates, following the strategy
of \cite{KM2008}.
 In fact, we will improve on the
results of \cite{CDP2017} by allowing exponents $\beta > \frac{d-1}{2}$ in the case of
bounded collision kernels (this was claimed without proof in \cite{CDP2017}).
It is straightforward (from the proof in this Appendix) to obtain the claimed
improvement in moments for the \emph{gain term} (only) in Proposition \ref{prop:KeyProp}
(i.e. $\delta > 0$).

The proof presented in this Appendix is adapted from an early manuscript of
\cite{CDP2017}. However, the proof presented here diverges from that of
\cite{CDP2017} in many details;
in particular, only constant or bounded collision kernels are considered here
(with a corresponding improvement in the available range of regularity in the $\beta$ exponent).
In fact we include only the case of \emph{constant} collision kernel in this Appendix;
the only difference with the bounded case is that we bound $\mathbf{b}$ by
its $L^\infty$ norm after passing to the spacetime Fourier transform.

\begin{proof} (case $\delta = 0$)

\textbf{Sobolev Estimates for the Loss Term}

Consider the loss term, which is as follows for a constant
collision kernel:
\begin{equation}
B^- [\gamma_1,\gamma_2](x,x^\prime) =
\gamma_1 (x,x^\prime) 
\gamma_2 \left( \frac{x+x^\prime}{2},\frac{x+x^\prime}{2}\right)
\end{equation}
We will fix some \emph{initial data}  $\gamma_1 (x,x^\prime)$,
$\gamma_2 (x,x^\prime)$,  and
consider the following function (it is the action of the nonlinearity upon
the free Schr{\" o}dinger flow):
\begin{equation}
B^- \left[ e^{\frac{1}{2} i t \left( \Delta_x - \Delta_{x^\prime}\right)}
\gamma_1, e^{\frac{1}{2} i t \left( \Delta_x - \Delta_{x^\prime}\right)}
\gamma_2 \right] (t,x,x^\prime)
\end{equation}
The spacetime Fourier transform of a function $F (t,x,x^\prime)$ is
\begin{equation}
\tilde{F} (\tau,\xi,\xi^\prime) =
\int dt dx dx^\prime e^{-i t \tau}
e^{-i x\cdot \xi} e^{-i x^\prime \cdot \xi^\prime}
F(t,x,x^\prime)
\end{equation}
The spacetime Fourier transform of 
$e^{\frac{1}{2} i t (\Delta_x - \Delta_{x^\prime})} \gamma_0$ is
\begin{equation}
\hat{\gamma_0} ( \xi,\xi^\prime) 
\delta \left( \tau + \frac{1}{2} |\xi|^2 - \frac{1}{2} |\xi^\prime|^2\right)
\end{equation}
where
\begin{equation}
\hat{\gamma}_0 (\xi,\xi^\prime) =
\int  dx dx^\prime
e^{-i x\cdot \xi} e^{-i x^\prime \cdot \xi^\prime}
\gamma_0 (x,x^\prime)
\end{equation}
We also have
\begin{equation}
\begin{aligned}
&\left( B^-\left[ e^{\frac{1}{2} i t \left( \Delta_x - \Delta_{x^\prime}
\right)} \gamma_1, e^{\frac{1}{2} i t 
\left( \Delta_x - \Delta_{x^\prime}\right)}
\gamma_2 \right] \right)^{\sim} (\tau,\xi,\xi^\prime) = \\
& = \int d\eta d\eta^\prime
\delta \left( \tau + \frac{1}{2} \left| \xi - \frac{\eta + \eta^\prime}{2}
\right|^2 - \frac{1}{2} \left| \xi^\prime - 
\frac{\eta + \eta^\prime}{2}\right|^2 +
\frac{1}{2} |\eta|^2  - \frac{1}{2} |\eta^\prime|^2 \right)\times \\
& \qquad \qquad\times \hat{\gamma_1} \left(
\xi - \frac{\eta+\eta^\prime}{2},
\xi^\prime - \frac{\eta+\eta^\prime}{2}\right)
\hat{\gamma_2} (\eta,\eta^\prime)
\end{aligned}
\end{equation}

We want to estimate the following integral, for suitable 
$\alpha,\beta > 0$:
\begin{equation}
\begin{aligned}
I_{\alpha,\beta}^- & = \int \left< \xi + \xi^\prime \right>^{2\alpha}
\left< \xi - \xi^\prime \right>^{2\beta}\times \\
& \times 
\left|\left(B^-\left[ e^{\frac{1}{2} i t \left( \Delta_x - \Delta_{x^\prime}\right)}
\gamma_1, e^{\frac{1}{2} i t \left( \Delta_x - \Delta_{x^\prime}\right)}
\gamma_2 \right] (\tau,\xi,\xi^\prime) \right)^\sim \right|^2 d\tau d\xi d\xi^\prime
\end{aligned}
\end{equation}
To start, observe that
\begin{equation*}
\begin{aligned}
& I_{\alpha,\beta}^- \leq \int d\tau d\xi d\xi^\prime 
d\eta_1 d\eta_1^\prime d\eta_2 d\eta_2^\prime 
\left< \xi+\xi^\prime \right>^{2\alpha}
\left< \xi-\xi^\prime \right>^{2\beta} \times \\
& \times \delta \left( \tau + \frac{1}{2} \left| \xi -
\frac{\eta_1+\eta_1^\prime}{2}\right|^2 - \frac{1}{2}
\left| \xi^\prime - \frac{\eta_1+\eta_1^\prime}{2}\right|^2
+\frac{1}{2}\left|\eta_1\right|^2 - \frac{1}{2}\left|\eta_1^\prime\right|^2
\right)\times \\
& \times \delta \left(  \tau + \frac{1}{2} \left| \xi -
\frac{\eta_2+\eta_2^\prime}{2}\right|^2 - \frac{1}{2}
\left| \xi^\prime - \frac{\eta_2+\eta_2^\prime}{2}\right|^2
+\frac{1}{2}\left|\eta_2\right|^2 - \frac{1}{2}\left|\eta_2^\prime\right|^2
\right)\times \\
& \times \left| \hat{\gamma_1}\left(\xi-\frac{\eta_1+\eta_1^\prime}{2},
\xi^\prime-\frac{\eta_1+\eta_1^\prime}{2}\right) \right|
\left| \hat{\gamma_2}(\eta_1,\eta_1^\prime)\right| \times\\
& \times \left| \hat{\gamma_1}\left(\xi-\frac{\eta_2+\eta_2^\prime}{2},
\xi^\prime-\frac{\eta_2+\eta_2^\prime}{2}\right) \right|
\left| \hat{\gamma_2}(\eta_2,\eta_2^\prime)\right| \\
\end{aligned}
\end{equation*}
Multiply and divide the integrand by the following factor:
\begin{equation}
\prod_{j=1}^2 \left\{
\left< \xi+\xi^\prime-\eta_j-\eta_j^\prime \right>^\alpha
\left< \xi-\xi^\prime \right>^\beta
\left<\eta_j+\eta_j^\prime\right>^\alpha 
\left<\eta_j-\eta_j^\prime \right>^\beta
\right\}
\end{equation}
Then group terms together and apply Cauchy-Schwarz \emph{pointwise}
under the integral sign. We obtain two different terms that are equal
due to symmetry under re-labeling coordinates; hence,
\begin{equation*}
\begin{aligned}
& I_{\alpha,\beta}^- \leq \int d\tau d\xi d\xi^\prime 
d\eta_1 d\eta_1^\prime d\eta_2 d\eta_2^\prime \times \\
& \times \frac{\left<\xi+\xi^\prime\right>^{2\alpha} 
\left<\xi-\xi^\prime\right>^{2\beta}}
{\left<\xi+\xi^\prime-\eta_1-\eta_1^\prime\right>^{2\alpha}
\left<\xi-\xi^\prime\right>^{2\beta}
\left<\eta_1+\eta_1^\prime\right>^{2\alpha} 
\left<\eta_1-\eta_1^\prime\right>^{2\beta}} \times \\
& \times \delta \left( \tau + \frac{1}{2} \left| \xi -
\frac{\eta_1+\eta_1^\prime}{2}\right|^2 - \frac{1}{2}
\left| \xi^\prime - \frac{\eta_1+\eta_1^\prime}{2}\right|^2
+\frac{1}{2}\left|\eta_1\right|^2 - \frac{1}{2}\left|\eta_1^\prime\right|^2
\right)\times \\
& \times \delta \left(  \tau + \frac{1}{2} \left| \xi -
\frac{\eta_2+\eta_2^\prime}{2}\right|^2 - \frac{1}{2}
\left| \xi^\prime - \frac{\eta_2+\eta_2^\prime}{2}\right|^2
+\frac{1}{2}\left|\eta_2\right|^2 - \frac{1}{2}\left|\eta_2^\prime\right|^2
\right)\times \\
& \times \left| \left< \xi+\xi^\prime-\eta_2-\eta_2^\prime\right>^\alpha
\left< \xi-\xi^\prime\right>^{\beta}
  \hat{\gamma_1}
\left(\xi-\frac{\eta_2+\eta_2^\prime}{2},
\xi^\prime-\frac{\eta_2+\eta_2^\prime}{2}\right) \right|^2\times \\
& \times \left| \left< \eta_2+\eta_2^\prime\right>^{\alpha}
\left< \eta_2 - \eta_2^\prime\right>^{\beta}
\hat{\gamma_2} (\eta_2,\eta_2^\prime)\right|^2 \\
\end{aligned}
\end{equation*}
The integral completely factorizes in the following way:
\begin{equation*}
\begin{aligned}
I_\alpha^- & \leq \int d\tau d\xi d\xi^\prime
\left( \int d\eta_1 d\eta_1^\prime  \dots\right)
\left( \int d\eta_2 d\eta_2^\prime \dots \right) \\
& \leq \left( \sup_{\tau,\xi,\xi^\prime} \int d\eta_1 d\eta_1^\prime
\dots \right) 
\times \int d\tau d\xi d\xi^\prime \left(
\int d\eta_2 d\eta_2^\prime \dots \right)
\end{aligned}
\end{equation*}
Finally we are able to conclude that if the following integral,
\begin{equation}
\label{eq:loss-integral}
\begin{aligned}
& \int d\eta d\eta^\prime \delta \left(\tau + 
\frac{1}{2} \left| \xi - \frac{\eta+\eta^\prime}{2}\right|^2 -
\frac{1}{2} \left| \xi^\prime - \frac{\eta+\eta^\prime}{2}\right|^2
+ \frac{1}{2} |\eta|^2 - \frac{1}{2} |\eta^\prime|^2 \right)\times \\
& \qquad \qquad \qquad \qquad \qquad \qquad \qquad
\times \frac{\left<\xi+\xi^\prime\right>^{2\alpha} }
{\left<\xi+\xi^\prime-\eta-\eta^\prime\right>^{2\alpha}
\left< \eta +\eta^\prime\right>^{2\alpha} 
\left<\eta-\eta^\prime\right>^{2\beta}}
\end{aligned}
\end{equation}
is bounded \emph{uniformly} with respect to $\tau, \xi, \xi^\prime$, then
the following estimate holds:
\begin{equation}
\left\Vert B^- \left[ 
e^{\frac{1}{2} i t \left( \Delta_x - \Delta_{x^\prime}\right)}
\gamma_1, e^{\frac{1}{2} i t \left( \Delta_x - \Delta_{x^\prime}\right)}
\gamma_2 \right] \right\Vert_{L^2_t 
H^{\alpha,\beta}} \leq
C \prod_{j=1,2} \left\Vert \gamma_j \right\Vert_{H^{\alpha,\beta}}
\end{equation}

Let us make the change of variables $w=\frac{\eta+\eta^\prime}{2}$,
$z=\frac{\eta-\eta^\prime}{2}$ in (\ref{eq:loss-integral}); then,
up to a constant, the integral becomes:
\begin{equation}
\label{eq:loss-integral-2}
\begin{aligned}
& \int dw dz \delta \left(\tau + 
\frac{1}{2} \left| \xi - w\right|^2 -
\frac{1}{2} \left| \xi^\prime - w\right|^2
+ \frac{1}{2} |w+z|^2 - \frac{1}{2} |w-z|^2 \right)\times \\
& \qquad \qquad \qquad \qquad \qquad \qquad
\times \frac{\left<\xi+\xi^\prime\right>^{2\alpha} }
{\left<\xi+\xi^\prime-2w\right>^{2\alpha}
\left< 2w \right>^{2\alpha} \left<2z\right>^{2\beta}}
\end{aligned}
\end{equation}
This is the same as:
\begin{equation}
\label{eq:loss-integral-2}
\begin{aligned}
&K = \int dw dz \delta \left(\tau + 
\frac{1}{2} \left( |\xi|^2 - |\xi^\prime|^2\right) -
\left( \xi - \xi^\prime - 2z\right)\cdot w \right)\times \\
& \qquad \qquad \qquad \qquad 
\times \frac{\left<\xi+\xi^\prime\right>^{2\alpha} }
{\left<\xi+\xi^\prime-2w\right>^{2\alpha}
\left< 2w \right>^{2\alpha} \left<2z\right>^{2\beta}}
\end{aligned}
\end{equation}
Hence, one way to parametrize the integral is to let $z\in\mathbb{R}^d$
be arbitrary and let $w$ range over a codimension one hyperplane
in $\mathbb{R}^d$; the hyperplane is determined by $\tau,\xi,\xi^\prime,z$.
Alternatively, we can let $w\in\mathbb{R}^d$ be arbitrary and let
$z$ range over a different codimension one hyperplane in $\mathbb{R}^d$.
We will choose the second option.

We have
\begin{equation}
\begin{aligned}
K = \left< \xi +\xi^\prime\right>^{2\alpha} 
\int_{\mathbb{R}^d} \frac{dw}{2|w| 
\left<\xi+\xi^\prime-2w\right>^{2\alpha}
\left<2w\right>^{2\alpha}}
\int_P \frac{dS(z)}{\left< 2z \right>^{2\beta}}
\end{aligned}
\end{equation}
where
\begin{equation}
P = \left\{ z \in \mathbb{R}^d \left|
\tau + \frac{1}{2} \left( |\xi|^2-|\xi^\prime|^2 \right)
- \left( \xi-\xi^\prime\right)\cdot w + 2w\cdot z = 0\right. \right\}
\end{equation}
and $dS (z)$ is the surface measure on $P$.

The integral over $P$ is no larger than the integral over a 
parallel hyperplane running through the origin, for which the
evaluation is very easy. We find that as long as $\beta > \frac{d-1}{2}$
the integral over $P$ converges, uniformly in $w,\xi,\xi^\prime,\tau$.
Hence we are left with
\begin{equation}
\begin{aligned}
K \lesssim \left< \xi +\xi^\prime\right>^{2\alpha} 
\int_{\mathbb{R}^d} \frac{dw}{2|w| 
\left<\xi+\xi^\prime-2w\right>^{2\alpha}
\left<2w\right>^{2\alpha}}
\end{aligned}
\end{equation}
The integral over the set $\left| w \right| \leq 1$ is trivially
bounded uniformly in $\xi,\xi^\prime$. Therefore the boundedness
of $K$ uniformly in $\xi,\xi^\prime$ is equivalent to the boundedness
of the follwing integral
\begin{equation}
K^\prime = \left< W \right>^{2\alpha} \int_{\mathbb{R}^d}
\frac{dw}{\left< w \right>^{2\alpha+1} \left< W-w \right>^{2\alpha}}
\end{equation}
uniformly in $W$.

The integral over $|w| < \frac{1}{2} |W|$ is trivially bounded uniformly
in $W$ if $\alpha > \frac{d-1}{2}$. Similarly the integral over
$|w|> 2|W|$ is bounded uniformly in $W$ if $\alpha > \frac{d-1}{2}$.
Hence we are left with the following integral:
\begin{equation}
 \left< W \right>^{2\alpha} \int_{\frac{1}{2} |W| \leq |w| \leq 2|W|}
\frac{dw}{\left< w \right>^{2\alpha+1} \left< W-w \right>^{2\alpha}}
\end{equation}
which is obviously bounded by the following integral:
\begin{equation}
\left< W \right>^{-1} \int_{\frac{1}{2} |W| \leq |w| \leq 2|W|}
\frac{dw}{\left< W-w\right>^{2\alpha}}
\end{equation}
Shifting $w$ to $W-w$ this is bounded by
\begin{equation}
\left< W \right>^{-1} \int_{|w| \leq 3|W|}
\frac{dw}{\left< w \right>^{2\alpha}}
\end{equation}
or even
\begin{equation}
\left< W \right>^{d-1} \int_{|u|\leq 3}
\frac{du}{\left< |W| u \right>^{2\alpha}}
\end{equation}
Splitting the last integral into the pieces $|u| \leq \frac{1}{|W|}$
and $\frac{1}{|W|} \leq |u| \leq 3$, we find that the integral $K^\prime$
is uniformly bounded in $W$ if $\alpha > \frac{d-1}{2}$.

Summarizing, we have the bound:
\begin{equation}
\left\Vert B^- \left[ 
e^{\frac{1}{2} i t \left( \Delta_x - \Delta_{x^\prime}\right)}
\gamma_1, e^{\frac{1}{2} i t \left( \Delta_x - \Delta_{x^\prime}\right)}
\gamma_2 \right] \right\Vert_{L^2_t 
H^{\alpha,\beta}} \leq
C \prod_{j=1,2} \left\Vert \gamma_j \right\Vert_{H^{\alpha,\beta}}
\end{equation}
as long as $\min (\alpha,\beta) > \frac{d-1}{2}$. The endpoint estimates
are not achieved, with respect to either $\alpha$ or $\beta$, by the
above argument. \\

\textbf{Sobolev Estimates for the Gain Term}

Consider the gain term, which is the following for a constant collision kernel:
\begin{equation}
B^+ [\gamma_1,\gamma_1](t,x,x^\prime) = \int_{\mathbb{S}^{d-1}}d\omega
B_\omega^+ [\gamma,\gamma](t,x,x^\prime)
\end{equation}
where
\begin{equation}
\begin{aligned}
& B_\omega^+ [\gamma_1,\gamma_2] (t,x,x^\prime) = \\
& =  \gamma_1 \left( t,x-\frac{1}{2} P_\omega (x-x^\prime),
x^\prime +\frac{1}{2} P_\omega (x-x^\prime)\right)\times\\
& \qquad \times \gamma_2 \left( t, \frac{x+x^\prime}{2} +
\frac{1}{2} P_\omega (x-x^\prime),\frac{x+x^\prime}{2}
-\frac{1}{2} P_\omega (x-x^\prime)\right)
\end{aligned}
\end{equation}
The spacetime Fourier transform of the function
\begin{equation}
B^+ 
\left[ e^{\frac{1}{2} i t \left( \Delta_x - \Delta_{x^\prime}\right)}
\gamma_1, e^{\frac{1}{2} i t \left( \Delta_x - \Delta_{x^\prime}\right)}
\gamma_2 \right] (t,x,x^\prime)
\end{equation}
is the following, up to a constant:
\begin{equation}
\label{eq:spacetime-fourier}
\begin{aligned}
& \int_{\mathbb{S}^{d-1}} d\omega 
\int d\eta_1 d\eta_1^\prime d\eta_2 d\eta_2^\prime
\delta \left( \tau + \frac{1}{2} |\eta_1|^2 - \frac{1}{2} 
|\eta_1^\prime|^2 + \frac{1}{2} |\eta_2|^2 -
\frac{1}{2} |\eta_2^\prime|^2 \right) \times \\
& \times \delta \left( -\xi + \eta_1 
+ \frac{\eta_2+\eta_2^\prime}{2}
-\frac{1}{2} P_\omega (\eta_1-\eta_1^\prime) 
+\frac{1}{2} P_\omega (\eta_2-\eta_2^\prime)\right)\times\\
& \times \delta \left( -\xi^\prime + \eta_1^\prime +
\frac{\eta_2+\eta_2^\prime}{2}
+ \frac{1}{2} P_\omega (\eta_1-\eta_1^\prime) 
- \frac{1}{2} P_\omega (\eta_2 - \eta_2^\prime)\right) \times \\
& \times \hat{\gamma}_1 (\eta_1,\eta_1^\prime)
\hat{\gamma}_2 (\eta_2,\eta_2^\prime)
\end{aligned}
\end{equation}
Introduce the change of variables $w_1 = \frac{\eta_1+\eta_1^\prime}{2}$,
$z_1 = \frac{\eta_1 - \eta_1^\prime}{2}$, 
$w_2 = \frac{\eta_2 + \eta_2^\prime}{2}$,
$z_2 = \frac{\eta_2-\eta_2^\prime}{2}$.
Then (\ref{eq:spacetime-fourier}) becomes
\begin{equation}
\label{eq:spacetime-fourier-2}
\begin{aligned}
& \int_{\mathbb{S}^{d-1}} d\omega \int dw_1 dz_1 dw_2 dz_2\times \\
&  \times \delta \left( \tau + \frac{1}{2} |w_1+z_1|^2 - \frac{1}{2} 
|w_1-z_1|^2 + \frac{1}{2} |w_2+z_2|^2 -
\frac{1}{2} |w_2-z_2|^2 \right) \times \\
& \times \delta \left( -\xi + w_1+z_1+ w_2
-  P_\omega (z_1-z_2) \right)\times\\
& \times \delta \left( -\xi^\prime + w_1 - z_1 + w_2
+ P_\omega (z_1 -z_2) \right) \times \\
& \times \hat{\gamma}_1 (w_1+z_1,w_1-z_1)
\hat{\gamma}_2 (w_2+z_2,w_2-z_2)
\end{aligned}
\end{equation}
Introduce yet another change of variables
$r_1 = \frac{w_1+w_2}{2}$, $s_1 = \frac{w_1-w_2}{2}$,
$r_2 = \frac{z_1+z_2}{2}$, $s_2 = \frac{z_1-z_2}{2}$.
Then (\ref{eq:spacetime-fourier-2}) becomes
\begin{equation}
\begin{aligned}
& \int_{\mathbb{S}^{d-1}} d\omega \int dr_1 ds_1 dr_2 ds_2 \times \\
&  \times \delta \left( \tau + \frac{1}{2} |r_1+s_1+r_2+s_2|^2 - \frac{1}{2} 
|r_1+s_1-r_2-s_2|^2 + \right. \\
& \left. \qquad + \frac{1}{2} |r_1-s_1+r_2-s_2|^2 -
\frac{1}{2} |r_1-s_1-r_2+s_2|^2 \right) \times \\
& \times \delta \left( -\xi + 2r_1 + r_2 + s_2
-  2 P_\omega s_2 \right)\times\\
& \times \delta \left( -\xi^\prime + 2r_1 - r_2 - s_2
+ 2 P_\omega s_2  \right) \times \\
& \times \hat{\gamma}_1 (r_1+s_1+r_2+s_2,r_1+s_1-r_2-s_2)\times \\
& \times \hat{\gamma}_2 (r_1-s_1+r_2-s_2,r_1-s_1-r_2+s_2)
\end{aligned}
\end{equation}
Replace $r_1$ with $\frac{r_1}{2}$ throughout:
\begin{equation}
\begin{aligned}
& \int_{\mathbb{S}^{d-1}} d\omega \int dr_1 ds_1 dr_2 ds_2 \times \\
&  \times \delta \left( \tau + \frac{1}{2} 
\left|\frac{r_1}{2}+s_1+r_2+s_2\right|^2 - 
\frac{1}{2} 
\left|\frac{r_1}{2}+s_1-r_2-s_2\right|^2 + \right. \\
& \left. \qquad + \frac{1}{2} \left|\frac{r_1}{2}-s_1+r_2-s_2\right|^2 -
\frac{1}{2} \left|\frac{r_1}{2}-s_1-r_2+s_2\right|^2 \right) \times \\
& \times \delta \left( -\xi + r_1 + r_2 + s_2
-  2 P_\omega s_2 \right)\times\\
& \times \delta \left( -\xi^\prime + r_1 - r_2 - s_2
+ 2 P_\omega s_2  \right) \times \\
& \times \hat{\gamma}_1 \left(\frac{r_1}{2}+s_1+r_2+s_2,
\frac{r_1}{2}+s_1-r_2-s_2\right)\times \\
& \times \hat{\gamma}_2 \left(\frac{r_1}{2}-s_1+r_2-s_2,
\frac{r_1}{2}-s_1-r_2+s_2\right)
\end{aligned}
\end{equation}
Finally perform the change of variables
$\zeta_1 = r_1+r_2$, $\zeta_2 = r_1-r_2$:
\begin{equation}
\begin{aligned}
& \int_{\mathbb{S}^{d-1}} d\omega \int d\zeta_1 d\zeta_2 ds_1 ds_2 \times \\
&  \times \delta \left( \tau + \frac{1}{2} 
\left|\frac{3\zeta_1}{4}-\frac{\zeta_2}{4}+s_1+s_2\right|^2 - 
\frac{1}{2} 
\left|-\frac{\zeta_1}{4}+\frac{3\zeta_2}{4}+s_1-s_2\right|^2 + \right. \\
& \left. \qquad + \frac{1}{2} 
\left|\frac{3\zeta_1}{4}-\frac{\zeta_2}{4}-s_1-s_2\right|^2 -
\frac{1}{2} \left|-\frac{\zeta_1}{4}+\frac{3\zeta_2}{4}
-s_1+s_2\right|^2 \right) \times \\
& \times \delta \left( -\xi + \zeta_1 + s_2
-  2 P_\omega s_2 \right)\times\\
& \times \delta \left( -\xi^\prime + \zeta_2 - s_2
+ 2 P_\omega s_2  \right) \times \\
& \times \hat{\gamma}_1 \left(\frac{3\zeta_1}{4}-\frac{\zeta_2}{4}+s_1+s_2,
-\frac{\zeta_1}{4}+\frac{3\zeta_2}{4}+s_1-s_2\right)\times \\
& \times \hat{\gamma}_2
\left(\frac{3\zeta_1}{4}-\frac{\zeta_2}{4}-s_1-s_2,
-\frac{\zeta_1}{4}+\frac{3\zeta_2}{4}-s_1+s_2\right)
\end{aligned}
\end{equation}
Now we can integrate out the variables $\zeta_1,\zeta_2$ to obtain:
\begin{equation}
\begin{aligned}
& \int_{\mathbb{S}^{d-1}} d\omega \int ds_1 ds_2 \times \\
&  \times \delta \left( \tau + \frac{1}{2} 
\left|s_1 + 2  s_2^{\Vert} + \frac{3\xi-\xi^\prime}{4}
\right|^2 - 
\frac{1}{2} 
\left|s_1-2 s_2^{\Vert} + \frac{3\xi^\prime-\xi}{4}
\right|^2 + \right. \\
& \left. \qquad \qquad + \frac{1}{2} 
\left|-s_1 - 2 s_2^{\bot} +
\frac{3\xi-\xi^\prime}{4}\right|^2 -
\frac{1}{2} \left|
-s_1 + 2 s_2^{\bot} + 
\frac{3\xi^\prime-\xi}{4} \right|^2 \right) \times \\
& \times \hat{\gamma}_1 
\left(s_1 + 2 s_2^{\Vert} + \frac{3\xi-\xi^\prime}{4},
s_1-2 s_2^{\Vert} + \frac{3\xi^\prime-\xi}{4}
\right)\times \\
& \times \hat{\gamma}_2 \left(-s_1 - 2 s_2^{\bot}+ \frac{3\xi-\xi^\prime}{4},
-s_1 + 2 s_2^{\bot} +  \frac{3\xi^\prime-\xi}{4} \right) 
\end{aligned}
\end{equation}
where $s_2^{\Vert} = P_\omega s_2$ and
$s_2^{\bot} = s_2 - P_\omega s_2$.

We want to estimate the following integral, for suitable $\alpha,\beta > 0$:
\begin{equation}
\begin{aligned}
I_{\alpha,\beta}^+ & = \int \left< \xi+\xi^\prime \right>^{2\alpha}
\left< \xi- \xi^\prime \right>^{2\beta}\times \\
& \times  \left|\left(B^+
\left[ e^{\frac{1}{2} i t \left( \Delta_x - \Delta_{x^\prime}\right)}
\gamma_1, e^{\frac{1}{2} i t \left( \Delta_x - \Delta_{x^\prime}\right)}
\gamma_2 \right] (\tau,\xi,\xi^\prime)\right)^\sim \right|^2 d\tau d\xi d\xi^\prime
\end{aligned}
\end{equation}
Reasoning as for the loss term, if we can show that the following integral
\begin{equation}
\label{eq:gain-int}
\begin{aligned}
& \int_{\mathbb{S}^{d-1}} d\omega \int ds_1 ds_2 \times \\
&  \times \delta \left( \tau + \frac{1}{2} 
\left|s_1 + 2  s_2^{\Vert} + \frac{3\xi-\xi^\prime}{4}
\right|^2 - 
\frac{1}{2} 
\left|s_1-2 s_2^{\Vert} + \frac{3\xi^\prime-\xi}{4}
\right|^2 + \right. \\
& \left. \qquad \qquad + \frac{1}{2} 
\left|-s_1 - 2 s_2^{\bot} +
\frac{3\xi-\xi^\prime}{4}\right|^2 -
\frac{1}{2} \left|
-s_1 + 2 s_2^{\bot} + 
\frac{3\xi^\prime-\xi}{4} \right|^2 \right) \times \\
& \times \frac{\left<\xi+\xi^\prime\right>^{2\alpha} 
\left<\xi-\xi^\prime\right>^{2\beta}}
{\left<2s_1+\frac{\xi+\xi^\prime}{2}\right>^{2\alpha}
\left<4s_2^{\Vert} + \xi - \xi^\prime\right>^{2\beta}
\left< -2s_1 + \frac{\xi+\xi^\prime}{2}\right>^{2\alpha}
\left< -4s_2^{\bot}+\xi-\xi^\prime\right>^{2\beta}}
\end{aligned}
\end{equation}
is bounded uniformly in $\tau,\xi,\xi^\prime$, then we will have the
following estimate:
\begin{equation}
\left\Vert B^+ \left[ 
e^{\frac{1}{2} i t \left( \Delta_x - \Delta_{x^\prime}\right)}
\gamma_1, e^{\frac{1}{2} i t \left( \Delta_x - \Delta_{x^\prime}\right)}
\gamma_2 \right] \right\Vert_{L^2_t 
H^{\alpha,\beta}} \leq
C \prod_{j=1,2}
\left\Vert \gamma_j \right\Vert_{H^{\alpha,\beta}}
\end{equation}

The integral (\ref{eq:gain-int}) is equivalent to the following integral:
\begin{equation}
\begin{aligned}
& \int_{\mathbb{S}^{d-1}} d\omega \int ds_1 ds_2\times \\
& \times \delta \left( \tau + \frac{1}{2} \left( |\xi|^2 - |\xi^\prime|^2
\right) +
\left( 4 s_1 - R_\omega ( \xi+\xi^\prime)\right)\cdot s_2 \right) \times \\
& \times  \frac{\left<\xi+\xi^\prime\right>^{2\alpha} 
\left<\xi-\xi^\prime\right>^{2\beta}}
{\left<2s_1+\frac{\xi+\xi^\prime}{2}\right>^{2\alpha}
\left<4s_2^{\Vert} + \xi - \xi^\prime\right>^{2\beta}
\left< -2s_1 + \frac{\xi+\xi^\prime}{2}\right>^{2\alpha}
\left< -4s_2^{\bot}+\xi-\xi^\prime\right>^{2\beta}}
\end{aligned}
\end{equation}
where $R_\omega (u) = u - 2 P_\omega u$ is reflection about
the plane perpendicular to $\omega$. This is in turn equivalent to the
following integral:
\begin{equation}
\begin{aligned}
& \int_{\mathbb{S}^{d-1}} d\omega \int ds_2
\frac{\left<\xi-\xi^\prime\right>^{2\beta}}
{\left| 4 s_2\right|
\left< 4s_2^{\Vert}+\xi-\xi^\prime  \right>^{2\beta}
\left< -4s_2^{\bot}+\xi-\xi^\prime  \right>^{2\beta}}\times \\
& \qquad \qquad \qquad \qquad\qquad
\times \int_P dS(s_1) \frac{\left<\xi+\xi^\prime\right>^{2\alpha}}
{\left< 2s_1 + \frac{\xi+\xi^\prime}{2} \right>^{2\alpha}
\left< -2s_1 + \frac{\xi+\xi^\prime}{2}  \right>^{2\alpha}}
\end{aligned}
\end{equation}
where $P \subset \mathbb{R}^d$ is the following codimension one hyperplane:
\begin{equation}
P = \left\{ s_1 \in \mathbb{R}^d \left|
\tau + \frac{1}{2} \left( |\xi|^2-|\xi^\prime|^2\right) +
\left( 4s_1-R_\omega (\xi+\xi^\prime)\right)\cdot s_2 = 0\right.\right\}
\end{equation}
Therefore we only need to show the boundedness of the following two
quantities uniformly in $\xi,\xi^\prime,\tau$:
\begin{equation}
I_1 = \sup_{P \subset \mathbb{R}^d : \dim P = d-1}
\int_P dS(s) \frac{\left<\xi+\xi^\prime\right>^{2\alpha}}
{\left<2s+\frac{\xi+\xi^\prime}{2}\right>^{2\alpha}
\left<-2s+\frac{\xi+\xi^\prime}{2}\right>^{2\alpha}}
\end{equation}
\begin{equation}
I_2 = \int_{\mathbb{S}^{d-1}} d\omega \int_{\mathbb{R}^d} ds
\frac{\left<\xi-\xi^\prime\right>^{2\beta}}
{|4s|\left<4s^{\Vert}+\xi-\xi^\prime\right>^{2\beta}
\left<-4s^{\bot}+\xi-\xi^\prime\right>^{2\beta}}
\end{equation}

Let us first consider the integral $I_2$; here we will assume
that $\beta > \frac{d-1}{2}$. Clearly,
 $I_2$ is equivalent to the following quantity:
\begin{equation}
I_2 \lesssim \int_{\mathbb{S}^{d-1}} d\omega \int_{\mathbb{R}^d} ds
\frac{\left<\xi-\xi^\prime\right>^{2\beta}}
{|s|\left<s^{\Vert}+\xi-\xi^\prime\right>^{2\beta}
\left<s^{\bot}+ \xi-\xi^\prime \right>^{2\beta}}
\end{equation}
Setting $W = \xi - \xi^\prime$, this gives:
\begin{equation}
I_2 \lesssim \int_{\mathbb{S}^{d-1}} d\omega \int_{\mathbb{R}^d} ds
\frac{\left<W\right>^{2\beta}}
{|s|\left<s^{\Vert}+W\right>^{2\beta}
\left<s^{\bot}+W\right>^{2\beta}}
\end{equation}
Moreover, since the integral for $|s|\leq 1$ is obviously uniformly
bounded in $W$, we may instead bound the following integral:
\begin{equation}
I_2^\prime \lesssim \int_{\mathbb{S}^{d-1}} d\omega \int_{\mathbb{R}^d} ds
\frac{\left<W\right>^{2\beta}}
{\left<s\right>\left<s^{\Vert}+W\right>^{2\beta}
\left<s^{\bot}+W\right>^{2\beta}}
\end{equation}
Since $|s^\Vert| \leq |s|$ we have:
\begin{equation}
I_2^\prime \lesssim \int_{\mathbb{S}^{d-1}} d\omega \int_{\mathbb{R}^d} ds
\frac{\left<W\right>^{2\beta}}
{\left<s^\Vert\right>\left<s^{\Vert}+W\right>^{2\beta}
\left<s^{\bot}+W\right>^{2\beta}}
\end{equation}
Therefore, for all large enough $|W|$,
\begin{equation}
\begin{aligned}
I_2^\prime & \lesssim \int_{\mathbb{S}^{d-1}} d\omega \int_{\mathbb{R}^d} ds
\frac{\left<W\right>^{2\beta}}
{\left< s^\Vert \right>
\left<s^{\Vert}+W\right>^{2\beta}
\left<s^{\bot}+W\right>^{2\beta}} \\
& = \int_{\mathbb{S}^{d-1}} d\omega \left< W \right>^{2\beta}
\left( \int \frac{ds^\Vert}
{\left< s^\Vert \right>
\left< s^\Vert + W \right>^{2\beta}}\right)
\left( \int \frac{ds^\bot}
{\left< s^\bot + W \right>^{2\beta}}\right) \\
& \lesssim \int_{\mathbb{S}^{d-1}} d\omega
\left< W \right>^{2\beta} \left(
\left< W\right>^{-1} \left<W^\bot\right>^{1-2\beta}
\log \left<W\right> \right)
\left( \left<W^\Vert\right>^{d-1-2\beta}\right)\\
\end{aligned}
\end{equation}
Note that the integral over $s^\bot$ is estimated by a trivial
computation, whereas the integral over $s^\Vert$ may be estimated
by considering separately the regions $|s^\Vert| \leq \frac{1}{2} |W|$,
$|s^\Vert| > 2|W|$, and $\frac{1}{2} |W| < |s^\Vert| \leq 2|W|$. The
integral over $|s^\Vert| \leq \frac{1}{2} |W|$ yields the logarithmic
divergence, whereas the integral over
$\frac{1}{2} |W| < |s^\Vert| \leq 2|W|$ provides the explicit
dependence on $W^\bot$.

We find that, for any fixed $\varepsilon > 0$, $I_2^\prime$ obeys the following
estimate:
\begin{equation}
I_2^\prime \lesssim
 \int_{\mathbb{S}^{d-1}} d\omega
\left< W \right>^{2\beta-1+\varepsilon}
 \left<W^\bot\right>^{1-2\beta}
\left<W^\Vert\right>^{d-1-2\beta}
\end{equation}
Then we have
\begin{equation}
\left< W \right>^{2\beta-1+\varepsilon} \lesssim
\left< W^\Vert \right>^{2\beta-1+\varepsilon} +
\left< W^\bot \right>^{2\beta-1+\varepsilon}
\end{equation}
Hence $I_2^\prime \lesssim I_2^{\prime \prime} + 
I_2^{\prime \prime \prime}$ where
\begin{equation}
I_2^{\prime \prime} = \int_{\mathbb{S}^{d-1}} d\omega
\left< W^\bot\right>^{1-2\beta}
\left< W^\Vert\right>^{d-2+\varepsilon}
\end{equation}
\begin{equation}
I_2^{\prime \prime \prime} = \int_{\mathbb{S}^{d-1}}
d\omega \left< W^\bot\right>^{\varepsilon}
\left< W^\Vert \right>^{d-1-2\beta}
\end{equation}
Then for any $\beta > \frac{d-1}{2}$, for some sufficiently small
$\varepsilon > 0$ depending on $\beta$, both $I_2^{\prime \prime}$
and $I_2^{\prime \prime \prime}$ may be bounded using dyadic
decompositions in the angular parameter $\omega$, as follows:
neglecting additive constants,
\begin{equation}
\begin{aligned}
I_2^{\prime \prime} & \lesssim \sum_{k=1}^\infty
\int_{\omega : 2^{-k-1} |W^\Vert| \leq |W^\bot| < 2^{-k} |W^\Vert|}
d\omega \left< W^\bot\right>^{1-2\beta}
\left< W^\Vert\right>^{d-2+\varepsilon} \\
& \lesssim \sum_{k=1}^\infty 2^{-k-1} \times
(2^{-k})^{d-2} \times (2^{k+1})^{d-2+\varepsilon} <\infty
\end{aligned}
\end{equation}
\begin{equation}
\begin{aligned}
I_2^{\prime \prime \prime} & \lesssim \sum_{k=1}^\infty
\int_{\omega : 2^{-k-1} |W^\bot| \leq |W^\Vert| < 2^{-k} |W^\bot|}
d\omega \left< W^\bot\right>^{\varepsilon}
\left< W^\Vert\right>^{d-1-2\beta} \\
& \lesssim \sum_{k=1}^\infty 2^{-k-1} \times (2^{k+1})^{\varepsilon} <\infty
\end{aligned}
\end{equation}
The factor of $(2^{-k})^{d-2}$ in $I_2^{\prime \prime}$ comes from
the Jacobian for spherical coordinates in $\mathbb{R}^d$.

We now turn to $I_1$, which is clearly bounded by the following
quantity:
\begin{equation}
I_1 \lesssim \sup_{W\in\mathbb{R}^d} 
\sup_{P\subset \mathbb{R}^d : \dim P = d-1}
\int_{P} dS(s) \frac{\left<W\right>^{2\alpha}}
{\left<s\right>^{2\alpha} \left<s+W\right>^{2\alpha}}
\end{equation}
The integrals over $P\cap \left\{ |s|<\frac{1}{2} |W|\right\}$,
 $P\cap \left\{ |s| > 2 |W|\right\}$, and
$P \cap \left\{ \frac{1}{2} |W| \leq |s| \leq 2|W|\right\}$
 are each easily bounded uniformly
in $W$ as long as $\alpha > \frac{d-1}{2}$.

To conclude, for any parameters $\alpha,\beta$ such that
$\min (\alpha,\beta) > \frac{d-1}{2}$, we have the following estimate:
\begin{equation}
\left\Vert B^+ \left[ 
e^{\frac{1}{2} i t \left( \Delta_x - \Delta_{x^\prime}\right)}
\gamma_1, e^{\frac{1}{2} i t \left( \Delta_x - \Delta_{x^\prime}\right)}
\gamma_2 \right] \right\Vert_{L^2_t 
H^{\alpha,\beta}} \leq
C \prod_{j=1,2}
\left\Vert \gamma_j \right\Vert_{H^{\alpha,\beta}}
\end{equation}
The endpoint estimates are not achieved, with respect to either
$\alpha$ or $\beta$, by the above argument.

\end{proof}



\end{document}